\documentclass[11pt]{amsart}

\usepackage{amsthm}                   
\usepackage{amsmath}                  
\usepackage{amsfonts}                 
\usepackage{amssymb}                  
\usepackage{mathrsfs}                 
\usepackage{mathtools}
\usepackage{enumerate}
\usepackage{bbm}                      
\usepackage[margin=1in]{geometry}  
\usepackage{multido}                  

\theoremstyle{definition}
\newtheorem{defi}{Definition}[section]
\newtheorem{rem}[defi]{Remark}
\newtheorem{ex}[defi]{Example}

\theoremstyle{plain}
\newtheorem{thm}[defi]{Theorem}
\newtheorem{lemma}[defi]{Lemma}
\newtheorem{cor}[defi]{Corollary}
\newtheorem{prop}[defi]{Proposition}

\numberwithin{equation}{section}

\newcommand{\R}{\ensuremath{\mathbbm{R}}}     
\newcommand{\C}{\ensuremath{\mathbbm{C}}}     
\newcommand{\N}{\ensuremath{\mathbbm{N}}}     
\newcommand{\Z}{\ensuremath{\mathbbm{Z}}}     
\newcommand{\Q}{\ensuremath{\mathbbm{Q}}}     
\renewcommand{\H}{\ensuremath{\mathbbm{H}}}   
\newcommand{\J}{\ensuremath{\mathbbm{J}}}     
\renewcommand{\L}{\ensuremath{\mathbbm{L}}}   

\def\sst{\scriptstyle}
\def\lt{\left}
\def\rt{\right}

\newcommand{\Rd}{\ensuremath{\R^d}}                       
\renewcommand{\Re}[1]{\ensuremath{\mathrm{Re}\lt(#1\rt)}} 
\renewcommand{\Im}[1]{\ensuremath{\mathrm{Im}\lt(#1\rt)}} 
\newcommand{\set}[1]{\lt\{#1\rt\}}			                  
\newcommand{\abs}[1]{\lt|#1\rt|}		                      
\newcommand{\gen}[1]{\lt\langle #1\rt\rangle}             
\newcommand{\inn}[2]{\lt\langle #1,#2\rt\rangle}          
\newcommand{\nrq}{{|q|}^2}                                
\DeclareMathOperator{\lcm}{lcm}                           
\DeclareMathOperator{\OC}{OC} 
\DeclareMathOperator{\SOC}{SOC}
\DeclareMathOperator{\SO}{SO}
\DeclareMathOperator{\OG}{O}
\DeclareMathOperator{\AC}{AOC}
\DeclareMathOperator{\Is}{\ensuremath{E}}                 
\newcommand{\e}{\ensuremath{\mathbf{e}}}                  
\newcommand{\ii}{\ensuremath{\mathbf{i}}}                 
\renewcommand{\j}{\ensuremath{\mathbf{j}}}                
\renewcommand{\k}{\ensuremath{\mathbf{k}}}                

\makeatletter
\def\imod#1{\allowbreak\mkern10mu({\operator@font mod}\,\,#1)} 
\def\@setcopyright{}                                           
\def\serieslogo@{}
\makeatother

\newcommand{\G}{\ensuremath{\Gamma}}        
\renewcommand{\S}{\ensuremath{\Sigma}}      
\newcommand{\vep}{\ensuremath{\varepsilon}} 

\begin{document}
	\author[M.J.C.~Loquias]{Manuel Joseph C.~Loquias}
	\address{Institute of Mathematics, University of the Philippines Diliman, 1101 Quezon City, Philippines}
	\email[M.J.C.~Loquias]{mjcloquias@math.upd.edu.ph}

	\author[P.~Zeiner]{Peter Zeiner}
	\address{Fakult\"at f\"ur Mathematik, Universit\"at Bielefeld, Postfach 100131, 33501 Germany}
	\email[P.~Zeiner]{pzeiner@math.uni-bielefeld.de}

	\title[The coincidence problem for shifted lattices and crystallographic point packings]{The coincidence problem for shifted lattices and crystallographic point packings}

	\begin{abstract}
	    A coincidence site lattice is a sublattice formed by the intersection of a lattice $\G$ in $\Rd$ with the image of $\G$ under a linear isometry.  
			Such a linear isometry is referred to as a linear coincidence isometry of $\G$.  
			Here, we consider the more general case allowing any affine isometry.  
			Consequently, general results on coincidence isometries of shifted copies of lattices, and of crystallographic point packings are obtained.  
			In particular, we discuss the shifted square lattice and the diamond packing in detail.
	\end{abstract}

	\subjclass[2010]{Primary 52C07; Secondary 11H06, 82D25, 52C23}

	\keywords{coincidence site lattice, grain boundary, crystallographic point packing, multilattice, diamond lattice}

	\date{18 July 2014}

	\maketitle
	
	\section{Introduction and Outline}   

		It was Friedel in 1911 who first recognized the usefulness of coincidence site lattices (CSLs) in describing and classifying grain boundaries of 
		crystals~\cite{F11}.  Since then, CSLs have been an indispensable tool in the study of grain boundaries, twins, and interfaces~\cite{KW49,B70,WB71}. This 
		prompted various authors to examine the CSLs of cubic and hexagonal crystals~\cite{R66,GBW74,G74,GW85}.

		The advent of quasicrystals in 1984 triggered a renewed interest in CSLs.  This is because experimental evidence showed that quasicrystals, like ordinary
		crystals, exhibit multiple grains, twin relationships, and coincidence quasilattices~\cite{W93,WL95}.  This led to a more general and mathematical 
		treatment of the coincidence problem for lattices in~\cite{B97}.  
		
		Various results are now known about the coincidences of lattices and modules in dimensions at most four.  The coincidence problem for certain planar lattices 
		and modules was solved in~\cite{PBR96,B97} using factorization properties of cyclotomic integers.  For lattices and modules in dimensions three and four,
		quaternions have proven to be an appropriate tool~\cite{B97,Z05,BPR07,RL97,Z06,BZ08,BGHZ08,H08,HZ10}.  

		However, the mathematical treatment of the coincidence problem has been mostly restricted to linear coincidence isometries, whereas isometries containing a
		translational part have rarely been treated so far.  Nevertheless, general (affine) isometries are important in crystallography.  Indeed, the situation where one shifts 
		the two component crystals against each other was investigated in \cite{GC72,F85} and references therein.  

		Even though the idea of introducing a shift after applying a linear coincidence isometry has already been dealt with in the physical literature, not much can 
		be found in the mathematical literature where a systematic treatment of the subject is still missing.  Initial steps in this general direction have actually
		been made in the appendix of \cite{PBR96}.  There, the authors considered coincidence isometries about certain points that are not lattice or module points.
		For example, they determined the set of coincidence isometries about the center of a Delauney cell of the square lattice and 
		calculated the corresponding indices.  

		The present work started from the Ph.D. thesis of the first named author~\cite{L10} and extends results from~\cite{LZ10}.  Here, the notion of a CSL is extended 
		to intersections of two lattices that are related by any isometry.  Such intersections are referred to as affine coincidence site lattices (ACSLs), and the 
		isometries that generate these intersections as affine coincidence isometries.  Theorem~\ref{affisom} identifies the affine coincidence isometries of a lattice, while 
		Equation~\eqref{ACSL} gives the resulting intersections.  

		The succeeding discussion covers a related and special case: the coincidence problem for shifted lattices.  That is, after translating the lattice $\G$ by
		some vector $x$, and upon application of a linear isometry $R$ to the shifted lattice $x+\G$, its intersection with $x+\G$ is considered.  Theorem~\ref{OCxG} 
		asserts that the linear coincidence isometries of $x+\G$ are precisely those coincidence isometries $R$ of $\G$ that satisfy $Rx-x\in\G+R\G$.  Moreover, the 
		CSLs of the shifted lattice are merely translates of CSLs of the original lattice.  Hence, no new values of coincidence indices are obtained by shifting the 
		lattice, with some values disappearing or their multiplicity being changed.
		
		Similar to the approach in~\cite{PBR96,B97}, an extensive analysis of the coincidences of a shifted square lattice in Section~\ref{coincshiftsqlatt} is 
		achieved by identifying the square lattice with the ring of Gaussian integers.  The coincidence problem for a shifted square lattice is completely solved when
		the shift comprises an irrational component (Theorem~\ref{irrational}).  For the remaining case, that is, when the shift may be written as a quotient of two 
		Gaussian integers that are relatively prime, one can compute for the set of coincidence rotations of the shifted square lattice using a divisibility
		condition involving the denominator of the shift (Lemma~\ref{divrule}).  In both instances, the set of coincidence rotations of a shifted square lattice forms 
		a group.  An example is given where the set of coincidence isometries of a shifted square lattice is not a group.  
		 
		The latter part of this contribution is concerned with the coincidences of sets of points formed by the union of a lattice with a finite number of shifted
		copies of the lattice.  Such sets are referred to as crystallographic point packings~\cite{CS99,BG13} or multilattices (see~\cite{PZ98} and references therein).  
		This idea should be useful for crystals having multiple atoms per primitive unit cell~\cite{GP82,PV83}.  Theorem~\ref{genthm} gives the solution of the coincidence problem for
		crystallographic point packings.  Simply put, the linear coincidence isometries of a crystallographic point packing are exactly the coincidence isometries of the lattice that 
		generates the crystallographic point packing - only the resulting intersections and corresponding indices may vary.  This paves the way for the solution of the coincidence problem
		for the diamond packing given in Theorem~\ref{OCdiamond}.     
		
	\section{Linear coincidences of lattices}
		
		We start with the basic definitions and some known results on linear coincidence isometries of lattices.  Details can be seen, for instance, in~\cite{B97,BG13}.
 
		A discrete subset $\G$ of $\Rd$ is a \emph{lattice} if it is the \Z-span of $d$~linearly independent vectors
		$v_1,\ldots,v_d\in\Rd$ over $\R$.  The set $\set{v_1,\ldots,v_d}$ is called a \emph{basis} of $\G$, and $\G=\Z v_1\oplus\cdots\oplus \Z v_d$.  As a group,
		$\G$ is isomorphic to the free Abelian group of rank~$d$.  Alternatively, one can characterize a lattice in $\Rd$ as a discrete co-compact subgroup of $\Rd$.
		A subset $\G'$ of $\G$ is a \emph{sublattice} of $\G$ if $\G'$ is a subgroup of $\G$ of finite (group) index.  The index of $\G'$ in $\G$ may be interpreted geometrically -- 
		$[\G:\G']$ is the quotient of the volume of a fundamental domain of $\G'$ by the volume of a fundamental domain of $\G$.

		For a lattice $\G$ in $\Rd$, its \emph{dual lattice} or \emph{reciprocal lattice} $\G^*$ is defined by
		\[\G^{\ast}\vcentcolon=\{x\in\Rd:\inn{x}{y}\in\Z \text{ for all }y\in\G\},\]
		where $\inn{\cdot}{\cdot}$ denotes the standard scalar product in $\Rd$.  Given a sublattice $\G'$ of $\G$, $\G^{\ast}$ is a sublattice of $(\G')^{\ast}$ with 
		$[(\G')^{\ast}:\G^{\ast}]=[\G:\G']$ and $(\G')^{\ast}/\G^{\ast}\cong\G/\G'$ \cite[Lemma 2.3]{B97}.
		
		Two lattices $\G_1$ and $\G_2$ are said to be \emph{commensurate}, denoted $\G_1\sim\G_2$, if $\G_1\cap\G_2$ is a sublattice of both $\G_1$ and $\G_2$.  
		Commensurateness between lattices defines an equivalence relation~\cite[Proposition 2.1]{B97}.  Given two commensurate lattices $\G_1$ and $\G_2$, their
		\emph{sum} ${\G_1+\G_2}\vcentcolon=\set{x_1+x_2:x_1\in\G_1,x_2\in\G_2}$ is also a lattice.  In fact, the following equations hold: 
		${(\G_1\cap\G_2)}^{\ast}=\G_1^{\ast}+\G_2^{\ast}$ and ${(\G_1+\G_2)}^{\ast}=\G_1^{\ast}\cap\G_2^{\ast}$~\cite[Proposition 2.2]{B97}.

		An orthogonal transformation $R\in \OG(d)\vcentcolon=\OG(d,\R)$ is a \emph{linear coincidence isometry} of the lattice $\G$ in $\Rd$ if $\G\sim R\G$.  The 
		sublattice $\G(R)\vcentcolon=\G\cap R\G$ is called the \emph{coincidence site lattice} (CSL) of $\G$ generated by $R$, while the index of $\G(R)$ in $\G$, 
		$\S_{\G}(R)\vcentcolon=[\G:\G(R)]=[R\G:\G(R)]$, is referred to as the \emph{coincidence index of $R$ with respect to $\G$}.  If no confusion arises, we simply write 
		$\S(R)$ to denote the coincidence index of $R$.  Clearly, symmetries in the point group of $\G$, $P(\G)=\set{R\in \OG(d):R\G=\G}$, are precisely those 
		linear coincidence isometries $R$ of $\G$ with $\S(R)=1$.

		The set of linear coincidence isometries of a lattice $\G$ in $\Rd$ is denoted by $\OC(\G)$ while the set of coincidence rotations of $\G$, that is, 
		$\OC(\G)\cap \SO(d)$, is written as $\SOC(\G)$.  Since commensurateness of lattices is an equivalence relation, the set $\OC(\G)$ forms a group having $\SOC(\G)$ 
		as a subgroup~\cite[Theorem 2.1]{B97}.

	\section{Affine coincidences of lattices}\label{affcoinc}
		
		Let $\G$ be a lattice in $\Rd$.  A subset of $\G$ will be called a \emph{cosublattice of $\G$} if it is a coset $\ell+\G'$ of some sublattice $\G'$ of $\G$.  
		The \emph{index of a cosublattice $\ell+\G'$ of $\G$}, denoted by $[\G:\ell+\G']$, is defined as the index of the sublattice $\G'$ in $\G$.  This definition
		of index makes sense geometrically: a translation does not change the volume of the fundamental domains of $\G$ and $\G'$.

		Denote by $\Is(d)$ the group of isometries of $\Rd$.  An element of $\Is(d)$ shall be written as $(v,R)$, where $(v,R):x\mapsto v+R(x)$, with $R\in \OG(d)$ (the 
		linear part of $f$) and $v\in\R^d$ (the translational part of $f$). The definition below generalizes the concept of a linear coincidence isometry to an affine 
		coincidence isometry.

		\begin{defi}
			Let $\G$ be a lattice in $\R^d$ and $(v,R)\in\Is(d)$. Then $(v,R)$ is an \emph{affine coincidence isometry} of $\G$ if $\G\cap (v,R)\G$ contains a 
			cosublattice of $\G$.
		\end{defi}

		The set of affine coincidence isometries of $\G$ shall be denoted by $\AC(\G)$.  It is easy to see that $\AC(\G)$ contains the group 
		\[\OC(\G)=\AC(\G)\cap \OG(d)=\set{(v,R)\in \AC(\G):v=0}.\]

		The following lemma describes the intersection of two lattices that are related by some isometry.
		\begin{lemma}\label{affint}
			Let $\G\subseteq\Rd$ be a lattice and $(v,R)\in\Is(d)$.  If $v\in\ell+R\G$ for some $\ell\in\G$, then $\G\cap (v,R)\G=\ell+(\G\cap R\G)$.
		\end{lemma}
		 
		Lemma~\ref{affint} is easy to see since $\G\cap (\ell,R)\G=\ell+(\G\cap R\G)$.  It brings about the following characterization of an affine coincidence isometry of a lattice.

		\begin{thm}\label{affisom}
			Let $\G$ be a lattice in $\Rd$.  Then $(v,R)\in\Is(d)$ is an affine coincidence isometry of $\G$ if and only if $R\in \OC(\G)$ and $v\in\G+R\G$.
		\end{thm}
		\begin{proof}
			It follows from Lemma~\ref{affint} that if $R\in \OC(\G)$ and $v\in\G+R\G$ then $\G\cap(v,R)\G$ is a coset of $\G(R)$.  In the other direction, 
			let $(v,R)\in \AC(\G)$.  Since $\G\cap(v,R)\G\neq\varnothing$, one has $v\in\G+R\G$.  Lemma~\ref{affint} then implies that
			$[\G:\G\cap R\G]=[\G:\G\cap (v,R)\G]<\infty$.  This yields $\G\sim R\G$ and $R\in \OC(\G)$.
		\end{proof}

		Therefore, the set of affine coincidence isometries of $\G$ is given by
		\[\AC(\G)=\set{(v,R)\in\Is(d):R\in \OC(\G)\text{ and }v\in\G+R\G}.\]
		Moreover, if $(v,R)\in \AC(\G)$ with $v\in\ell+R\G$ for some $\ell\in\G$, then
		\begin{equation}\label{ACSL}
			\G\cap (v,R)\G=\ell+\G(R)
		\end{equation}
		by Lemma~\ref{affint}. Thus, $\G\cap (v,R)\G$ is a coset of $\G(R)$.  This means that the intersection $\G\cap (v,R)\G$ does not only contain a 
		cosublattice of $\G$ but is in fact a cosublattice of $\G$.  For this reason, we shall refer to $\G\cap (v,R)\G$ as an \emph{affine coincidence site lattice} 
		(ACSL) of $\G$.  In addition, each $R\in \OC(\G)$ corresponds to $\S(R)$ distinct possible ACSLs.

		\begin{rem}
			Another lattice of interest in the study of grain boundaries is the \emph{displacement shift complete} (DSC) \emph{lattice}.  It is the lattice formed by
			all possible displacement vectors that preserve the structure of the grain boundary.  In this setting, given a linear coincidence isometry $R$ of the
			lattice $\G$, the corresponding DSC lattice is $\set{v:(v,R)\in \AC(\G)}=\G+R\G$ by Theorem~\ref{affisom}.  This conclusion is in agreement with the main 
			result of~\cite{G74b}, which states that the DSC lattice generated by $R$ is the dual lattice of the CSL of $\G^{\ast}$ obtained from $R$, that is, 
			${\lt(\G^{\ast}\cap R\G^{\ast}\rt)}^{\ast}=\G+R\G$.
		\end{rem}

		Now, the identity isometry $\mathbbm{1}_d\in \AC(\G)$ for any lattice $\G$ in $\Rd$.  In addition, it follows from Theorem~\ref{affisom} that the inverse of
		every isometry in $\AC(\G)$ is also in $\AC(\G)$.  However, the product of two affine coincidence isometries of $\G$ may or may not be an element of $\AC(\G)$.  
		Thus, the set $\AC(\G)$ does not always form a group.  Actually, $\AC(\G)$ is a group only if it is sufficiently small.
		\begin{prop}\label{group}
			Let $\G\subseteq\Rd$ be a lattice.  Then $\AC(\G)$ is a group if and only if it is the symmetry group $G$ of $\G$.
		\end{prop}
		\begin{proof}
			Suppose $\AC(\G)$ is a group and take $(v,R)\in \AC(\G)$.   It follows from Theorem~\ref{affisom} that the product $(v,R)(0,R^{-1})=(v,\mathbbm{1}_d)\in \AC(\G)$ 
			and so $v\in\G$.  Furthermore, $\G+R\G=\G$ and hence, $R\in P(\G)$.  Since $\G$ is a lattice, its symmetry group $G$ must be symmorphic,
			i.e., it is the semidirect product of $P(\G)$ with its translation subgroup $T(G)$.  Thus, $(v,R)\in G$.  
		\end{proof}
		In particular, $\AC(\G)$ is a group only if $\OC(\G)=P(\G)$, i.e., if $\G$ has no coincidence isometries $R$ with $\S(R)>1$. 

		Note that Proposition~\ref{group} is only true for lattices.  It is difficult in the case of crystallographic point packings. 
		Without going into details here, we mention that $\AC(L)$ is a group only if it is a symmorphic space group. In fact, $P(L)$ has to be a holohedry and 
		$\AC(L)$ turns out to be the symmetry group of some suitable lattice $\varLambda\supseteq L$. Note that $\AC(L)$ may be a proper supergroup of the 
		symmetry group of $L$, where the latter may even be a non-symmorphic space group.

	\section{Linear coincidences of shifted lattices}\label{shiftedlatt}
		
		We now turn our attention to shifted copies $x+\G$ of a lattice $\G$ in $\Rd$ obtained by translating all the points of $\G$ by the vector $x\in\Rd$.  By a 
		\emph{cosublattice of the shifted lattice $x+\G$}, we mean a subset of $x+\G$ of the	form $x+(\ell+\G')$ where $\ell+\G'$ is a cosublattice of $\G$.  In 
		addition, the \emph{index of the cosublattice $x+(\ell+\G')$ in $x+\G$} is understood to be  $[x+\G:x+(\ell+\G')]\vcentcolon=[\G:\G']$.  There is no ambiguity here - 
		relabeling $x$ as the origin gives back the original lattice $\G$ and cosublattice $\ell+\G'$.  Of particular interest in this section are intersections of 
		the form $(x+\G)\cap R(x+\G)$, where $R\in \OG(d)$.

		\begin{defi}
			Let $\G$ be a lattice in $\Rd$ and $x\in\Rd$.  An $R\in \OG(d)$ is said to be a \emph{linear coincidence isometry of the shifted lattice} $x+\G$ if 
			$(x+\G)\cap R(x+\G)$ is a cosublattice of $x+\G$.
		\end{defi}
		
		The intersection $(x+\G)\cap R(x+\G)$ will also be referred to as a \emph{CSL of the shifted lattice $x+\G$}.  The \emph{coincidence index of $R$ with respect 
		to $x+\G$} is taken to be $\S_{x+\G}(R)\vcentcolon={[x+\G:(x+\G)\cap R(x+\G)] }$. The set of all linear coincidence isometries of $x+\G$ shall be denoted by $\OC(x+\G)$.  	
		Likewise, we take $\SOC(x+\G)\vcentcolon=\OC(x+\G)\cap \SO(d)$.		

		\begin{rem}\label{sitesymm}
			Observe that applying a linear isometry $R$ to the shifted lattice $x+\G$ is equivalent to applying the same isometry $R$ but with center at $-x$ to the 
			original lattice~$\G$.  Hence, just as $\OC(\G)$ is an extension of $P(\G)$, one may interpret $\OC(x+\G)$ as a generalization of the stabilizer of the point
			$-x$.
		\end{rem}

		The following theorem characterizes a linear coincidence isometry $R$ of a shifted lattice $x+\G$ and identifies the CSL of $x+\G$ generated by $R$.  The
		result lies on the fact that taking the intersection of $x+\G$ and $R(x+\G)$ corresponds to a shift of the intersection of $\G$ and $(Rx-x,R)\G$ by $x$.  It 
		is a special case of Lemma~\ref{intersectLgen} which will be stated and proved in Section~\ref{sectmultlatt}.
		\begin{thm}\label{OCxG}
			Let $\G$ be a lattice in $\Rd$ and $x\in\Rd$. Then 
			\[\OC(x+\G)=\set{R\in \OC(\G):Rx-x\in\G +R\G}.\]  
			In addition, if $R\in \OC(x+\G)$ with $Rx-x\in\ell+R\G$ for some $\ell\in\G$, then 
			\begin{equation}\label{ACSLshifted}
				(x+\G)\cap R(x+\G)=(x+\ell)+\G(R).
			\end{equation} 			
		\end{thm}
		
		Equation~\eqref{ACSLshifted} indicates that the CSL of the shifted lattice $x+\G$ generated by $R\in \OC(x+\G)$ is obtained by translating some coset of 
		$\G(R)$ in $\G$ by $x$.  	Consequently, 
		\begin{equation}\label{coincindCSLshifted}
			\S_{x+\G}(R)=\S_{\G}(R)
		\end{equation}
		for all $R\in \OC(x+\G)$.  This means that shifting a lattice does not yield any new values of coincidence indices. 

		Let $S\in P(\G)$.  If $R\in \OC(\G)$ then $RS\in \OC(\G)$ and the CSLs generated by $R$ and $RS$ are the same, that is, $\G(RS)=\G(R)$.  The corresponding 
		statement for linear coincidence isometries of shifted lattices reads as follows.  It will prove to be useful when counting the number of CSLs of a shifted 
		lattice for a given index.

		\begin{prop}\label{countCSLshift}
			Let $x+\G\subseteq\Rd$ be a shifted lattice, $S\in P(\G)$, and suppose that $R, RS\in \OC(x+\G)$.  Then $(x+\G)\cap RS(x+\G)=(x+\G)\cap R(x+\G)$ if and only 
			if $S\in \OC(x+\G)$.  In particular, if $\OC(x+\G)$ forms a group, then $(x+\G)\cap RS(x+\G)=(x+\G)\cap R(x+\G)$.
		\end{prop}
		\begin{proof}
			It follows from Theorem~\ref{OCxG} that $Rx-x\in\ell_1+R\G$ and $RSx-x\in\ell_2+R\G$ for some $\ell_1,\ell_2\in\G$.  Equation~\eqref{ACSLshifted} yields that
			$(x+\G)\cap RS(x+\G)=(x+\G)\cap R(x+\G)$ if and only if $Sx-x\in\G$.  Applying Theorem~\ref{OCxG} proves the claim.
		\end{proof}

		 Note that for any $S\in P(\G)$, the condition $S\in \OC(x+\G)$ in Proposition~\ref{countCSLshift} is equivalent to saying that $S$ is an element of the 
		 stabilizer of $-x$ (see Remark~\ref{sitesymm}).

		\begin{prop}\label{OCxGpoint}
			Let $\G\subseteq\Rd$ be a lattice and $x\in\Rd$.  If $S\in P(\G)$ then \[\OC(Sx+\G)=S[\OC(x+\G)]S^{-1}.\]
		\end{prop}
		\begin{proof}
			This is a consequence of Theorem~\ref{OCxG} because $SRS^{-1}(Sx)-Sx\in\G+SRS^{-1}\G$ if and only if $Rx-x\in\G+R\G$ for all $R\in \OC(\G)$.
		\end{proof}

		For a given lattice $\G\subseteq\Rd$, it is enough to consider values of $x$ in a fundamental domain of $\G$ to compute for all the different possible sets 
		$\OC(x+\G)$.  Proposition~\ref{OCxGpoint} asserts even more: it suffices to look at values of $x$ in a fundamental domain of the symmetry group of~$\G$.  
		
		Furthermore, the following inclusion property follows immediately from Theorem~\ref{OCxG}.
		\begin{lemma}\label{sumofx}
			If $\G$ is a lattice in $\R^d$ and $x,y\in\Rd$, then for all $a,b\in\Z$, \[\OC(x+\G)\cap \OC(y+\G)\subseteq \OC[(ax+by)+\G].\]
		\end{lemma}
		
		\begin{cor}
			Let $\G$ be a lattice in $\Rd$ and $x=(1/n)\ell$, where $\ell\in\G$ and $n\in\N$.  If $a\in\Z$ with $a$ and $n$ relatively prime, then
			$\OC(ax+\G)=\OC(x+\G)$. 
		\end{cor}
		\begin{proof}
			The inclusion $\OC(x+\G)\subseteq \OC(ax+\G)$ follows directly from Lemma~\ref{sumofx}.  Since $a$ is relatively prime to $n$, there exist integers $b$ and
			$c$ such that $ab+nc=1$.  
			Applying again Lemma~\ref{sumofx} yields \[\OC(ax+\G)\subseteq \OC[(ab+nc)(\tfrac{1}{n}\ell)+\G]=\OC(x+\G).\]
		\end{proof}

		The next proposition compares the sets of linear coincidence isometries of shifts of similar lattices and is the analogue of Lemma 2.5 in~\cite{B97} for
		shifted lattices.

		\begin{prop}\label{OCxGsim}
			Let $\G$ be a lattice in $\Rd$ and $x\in\Rd$.
			\begin{enumerate}[\rm(i)]
				\item If $\lambda\in\R^+$ then $\OC(\lambda x+\lambda\G)=\OC(x+\G)$ with $\S_{\lambda x+\lambda\G}(R)=\S_{\G}(R)$ for all $R\in \OC(\lambda x+\lambda\G)$.
				
				\item If $S\in \OG(d)$ then $\OC(Sx+S\G)=S[\OC(x+\G)]S^{-1}$ with $\S_{Sx+S\G}(R)=\S_{\G}(S^{-1}RS)$ for all $R\in \OC(Sx+S\G)$.
			\end{enumerate}
		\end{prop}
		\begin{proof}
			Both statements follow from Theorem~\ref{OCxG} and Equation \eqref{coincindCSLshifted}.
		\end{proof}

		Now, it is evident from Theorem~\ref{OCxG} that $\OC(x+\G)$ is a subset of $\OC(\G)$.  The set ${\OC(x+\G)}$ is certainly nonempty because it contains the
		identity isometry.  It also follows from Theorem~\ref{OCxG} that $\OC(x+\G)$ is closed under inverses, that is, $R^{-1}\in {\OC(x+\G)}$ whenever 
		$R\in \OC(x+\G)$.  However, given $R_1$, $R_2\in \OC(x+\G)$, the product $R_2R_1$ is not necessarily in ${\OC(x+\G)}$.  Thus, one obtains the following result.
		
		\begin{prop}\label{OCgroup}
			For a given lattice $\G\subseteq\Rd$ and $x\in\Rd$, the set $\OC(x+\G)$ is a group if and only if it is closed under composition.
		\end{prop}
		
		 We shall see in Example~\ref{exnotgroup} an instance when $\OC(x+\G)$ fails to form a group.  In any case, the product of two linear coincidence isometries of 
		 $x+\G$ whose coincidence indices are relatively prime turns out to be again a linear coincidence isometry of $x+\G$.  This result is stated in the next 
		 proposition.

		\begin{prop}\label{relprimeshifted}
			Let $\G\subseteq\Rd$ be a lattice and $x\in\Rd$.  If $R_1,R_2\in \OC(x+\G)$ with $\S(R_1)$ and $\S(R_2)$ relatively prime, then $R_2R_1\in \OC(x+\G)$.
		\end{prop}
		\begin{proof}
			From Theorem~\ref{OCxG}, $R_j\in \OC(\G)$ and $R_jx-x\in\G+R_j\G$ for $j\in\set{1,2}$.  Thus, the product $R_2R_1\in \OC(\G)$.  In addition, 
			$R_2R_1x-x\in\G+R_2R_1\G$ because $\S(R_1)$ and $\S(R_2)$ are relatively prime (see~\cite[Figure 2]{Z10}).  The claim now follows from Theorem~\ref{OCxG}.
		\end{proof}

	\section{Linear coincidences of a shifted square lattice}\label{coincshiftsqlatt}
	
		Let us illustrate our results for the square lattice $\Z^2\simeq \Z[i]$.  The solution of its ordinary coincidence problem is known in detail~\cite{B97,PBR96,LZ10} 
		and we can get very explicit results for its shifted copies as well.  Some of the results have already been published in~\cite{LZ10}, but for sake of completeness we
		will recall them here.

		\subsection{Solution of the coincidence problem for the square lattice}\label{sqlattice}
			
			Let us first summarize the coincidences of the square lattice $\Z^2$ (see~\cite{B97,PBR96,LZ10} for details).  We restrict our discussion to coincidence rotations 
			at the outset and later on extend it to include coincidence reflections.
	
			The group of coincidence rotations of $\Z^2$ is $\SO(2,\Q)$.  To determine the structure of this group, the square lattice is identified with the 
			ring of Gaussian integers \[\G=\Z[i]=\set{m+ni:m,n\in\Z,i^2=-1}\] embedded in $\C$.  It can be shown that every coincidence rotation in $\SOC(\G)$ by an 
			angle of $\theta$ in the counterclockwise direction corresponds to multiplication by the complex number $e^{i\theta}=\vep z/\overline{z}$ on the unit 
			circle, where $\vep\in\set{\pm 1,\pm i}$ is a unit in $\Z[i]$ and $z$ is a Gaussian integer with $z$ relatively prime to $\overline{z}$.  That is, a coincidence 
			rotation $R$ of $\G$ is equivalent to multiplication by the complex number 
			\begin{equation}\label{coincrot}
				\vep\cdot\prod_{p\equiv 1(4)}{\lt(\frac{\omega_p}{\overline{\omega_p}}\rt)}^{n_p},
			\end{equation}
			where $n_p\in\Z$ and only a finite number of $n_p\neq 0$, $p$ runs over all rational primes $p\equiv 1\imod{4}$ (called splitting primes in $\Z[i]$), and 
			$\omega_p$, and its complex conjugate $\overline{\omega_p}$, are the Gaussian prime factors of $p=\omega_p\cdot\overline{\omega_p}$.  Then $z$ reads
			\begin{equation}\label{num}
				z=\prod_{\stackrel{p\equiv 1(4)}{n_p>0}}{\omega_p}^{n_p}\cdot\prod_{\stackrel{p\equiv 1(4)}{n_p<0}}{{\lt(\overline{{\omega_p}}\rt)}^{\,-n_p}},
			\end{equation}
			and the coincidence index of $R$ is equal to the number theoretic norm of $z$, $\S(R)=N(z)\vcentcolon=z\cdot\overline{z}=\abs{z}^2$.  In addition, the CSL obtained 
			from $R$ is the principal ideal $\G(R)=(z)\vcentcolon=z\Z[i]$.  Consequently, the group of coincidence rotations of the square lattice is given by 
			$\SO(2,\Q)\cong C_4\times\Z^{(\aleph_0)}$, where $C_4$ is the cyclic group of order 4 generated by $i$, and $\Z^{(\aleph_0)}$ is the direct sum of 
			countably many infinite cyclic groups each of which is generated by some $\omega_p/\overline{\omega_p}$.
		
			Every coincidence reflection $T$ of $\Z^2$ can be written as $T=R\cdot T_r$, where $R\in \SOC(\G)$ and $T_r$ is the reflection along the real axis 
			(corresponding to complex conjugation).  Hence, $\S(T)=\S(R)$, $\G(T)=\G(R)$, and $\OC(\Z^2)=\OG(2,\Q)=\SOC(\Z^2)\rtimes\gen{T_r}$ (where $\rtimes$ stands for semidirect product).
	
			The coincidence indices and the number of CSLs of $\Z^2$ for a given index $m$ are described by means of a generating function.  If $f_{\Z^2}(m)$ denotes 
			the number of CSLs of $\Z^2$ of index~$m$, then $f_{\Z^2}$ is \emph{multiplicative} (that is, $f_{\Z^2}(1)=1$ and $f_{\Z^2}(mn)=f_{\Z^2}(m)f_{\Z^2}(n)$ 
			whenever $m$ and $n$ are relatively prime).  The generating function for $f_{\Z^2}$ as a Dirichlet series $\Phi_{\Z^2}(s)$ is given by
			\begin{equation}\label{dirsq}
				\begin{aligned}
					\Phi_{\Z^2}(s)&=\sum_{m=1}^{\infty}{\frac{f_{\Z^2}(m)}{m^s}}=\prod_{p\equiv 1(4)}{\frac{1+p^{-s}}{1-p^{-s}}}
					=\frac{1}{1+2^{-s}}\cdot\frac{\zeta_{\Q (i)}(s)}{\zeta(2s)}\\
					&=1+\tfrac{2}{5^s}+\tfrac{2}{13^s}+\tfrac{2}{17^s}+\tfrac{2}{25^s}+\tfrac{2}{29^s}+\tfrac{2}{37^s}+
					\tfrac{2}{41^s}+\tfrac{2}{53^s}+\tfrac{2}{61^s}+\tfrac{4}{65^s}+\tfrac{2}{73^s}+\cdots
				\end{aligned}
			\end{equation}
			where $\zeta_{\Q (i)}(s)$ is the Dedekind zeta function of the quadratic field $\Q(i)$ and $\zeta(s)=\zeta_{\Q}(s)$ is Riemann's zeta function (see 
			\cite{C78,W97}).  As the rightmost pole of $\Phi_{\Z^2}(s)$ is located at $s=1$, we can infer from Delange's theorem (see for instance, 
			\cite[Theorem 5 of Appendix]{BM99}) that the summatory function $\sum_{m\leq N}f_{\Z^2}(m)$ grows asymptotically as $N/\pi$.  In other words, the
			number of CSLs of $\Z^2$ of index at most $N$ is asymptotically given by $N/\pi$.

			The number of coincidence rotations of $\Z^2$ for a given index $m$ is given by $\hat{f}_{\Z^2}(m)=4f_{\Z^2}(m)$ and the Dirichlet series generating function for 
			$\hat{f}_{\Z^2}$ is $4\Phi_{\Z^2}(s)$.
	
			\begin{rem}\label{coincrotsq}
				Observe from the complex number in~\eqref{coincrot} and Equation~\eqref{num} that each coincidence rotation $R$ of $\G=\Z^2$ can be associated to a
				numerator $z$ and unit $\vep$, and this shall be written as $R_{z,\vep}$.  Note however that this correspondence is not unique (see~\cite{LZ10} for details).	
				Similarly, $T_{z,\vep}\in \OC(\G)\setminus \SOC(\G)$ is understood to be the coincidence reflection $T_{z,\vep}=R_{z,\vep}\cdot T_r$.
			\end{rem}

		\subsection{The sets $\SOC(x+\G)$ and $\OC(x+\G)$}
		
			It is well known that $\SOC(\G)$ and $\OC(\G)$ are groups for arbitrary lattices $\G$, but $\SOC(x+\G)$ and $\OC(x+\G)$ cannot be expected to form groups in general. 
			Hence, we first concentrate on determining the structure of $\SOC(x+\G)$ and $\OC(x+\G)$ for $\G=\Z[i]$.  To this end, we start with a criterion for $R\in OC(\G)$ to be a
			coincidence isometry of $x+\G$~\cite{LZ10} (see also~\cite{L10}).
			
			\begin{lemma}\label{RzvepSOC}
				Let $\G=\Z[i]$, $x\in\C$, $R_{z,\vep}\in \SOC(\G)$, and $T_{z,\vep}\in \OC(\G)\setminus \SOC(\G)$. Then 
				\begin{enumerate}[\rm(i)]
					\item $R_{z,\vep}\in \SOC(x+\G)$ if and only if $(\vep z-\overline{z})x\in\G$.
					
					\item $T_{z,\vep}\in \OC(x+\G)$ if and only if $\vep z\overline{x}-\overline{z}x\in\G$.
				\end{enumerate}
			\end{lemma}
			
			It turns out that the set of coincidence rotations of $x+\G$ forms a group (see~\cite[Theorem~3]{LZ10} or~\cite[Theorem 3.20]{L10}).
			\begin{thm}\label{SOCsqgroup}
				If $\G=\Z[i]$ then $\SOC(x+\G)$ is a subgroup of $\SOC(\G)$ for all $x\in\C$.
			\end{thm}
			The core of the proof is to show that the product $R_{z_2,\vep_1}R_{z_1,\vep_2}=R_{h_2h_1,\vep_2\vep_1}$ is again a coincidence rotation of $x+\G$, where 
			$g\vcentcolon=\gcd(\overline{z_1},z_2)$ (up to a factor that is a unit of $\Z[i]$) and $h_1=z_1/\overline{g}$, $h_2=z_2/g$. This is achieved by showing that 
			$\big(\vep_2\vep_1 h_2 h_1-\overline{h_2h_1}\,\big)x\in(1/g)\G\cap (1/\overline{g})\G$.		
			
			However, the situation is more complicated for $\OC(x+\G)$.  Analogous techniques allow us to show that the product of a rotation $R\in \SOC(x+\G)$ and a reflection 
			$T\in \OC(x+\G)$ is again in $\OC(x+\G)$, but they fail for the product of two reflections in $\OC(x+\G)$. Thus we get the following weaker result.
			\begin{lemma}\label{OCsq}
				Let $\G=\Z[i]$ and $x\in\C$.  Then $\OC(x+\G)$ is a subgroup of $\OC(\G)$ if and only if for any coincidence reflections $T_1$, $T_2\in \OC(x+\G)$, 
				the coincidence rotation $T_2T_1\in \SOC(x+\G)$.
			\end{lemma}
						
			\begin{rem}\label{remprodtworefsq}
				Let $x\in\C$ and $T_j=T_{z_j,\vep_j}\in \OC(x+\G)\setminus \SOC(x+\G)$ for $j\in\set{1,2}$.  Applying the procedure used in the proof of
				Theorem~\ref{SOCsqgroup} to the product $T_2T_1$ only leads to
				\begin{equation}\label{prodrefsq}
					\big(\vep_2\overline{\vep_1}h_2\overline{h_1}-\overline{h_2}h_1\big)x\in\tfrac{1}{\overline{g}}\G,
				\end{equation}
				where $g\vcentcolon=\gcd(z_1,z_2)$ and $z_j=h_jg$ for $j\in\set{1,2}$.  It follows then from Lemma~\ref{RzvepSOC} that if $z_1$ were relatively prime to $z_2$, then 
				$T_2T_1=R_{h_2\overline{h_1},\vep_2\overline{\vep_1}}\in (S)\OC(x+\G)$.  This fact can also be deduced from Proposition~\ref{relprimeshifted}, 
				because if $z_1$ and $z_2$ were relatively prime in $\Z[i]$, then so are $N(z_1)=\S(R_1)$ and $N(z_2)=\S(R_2)$.
			\end{rem}
	
			\begin{prop}\label{OCsqprop}
				Let $\G=\Z[i]$ and $x\in\C$.  If $\OC(x+\G)$ contains a reflection symmetry $T\in P(\G)$ then $\OC(x+\G)=\SOC(x+\G)\rtimes\langle T\rangle$ and is a subgroup 
				of $\OC(\G)$.  Otherwise, the coincidence reflection $T_{z,\vep}\notin \OC(x+\G)$ for all units $\vep$ of $\Z[i]$ whenever $R=R_{z,\vep'}\in \SOC(x+\G)$ for 
				some unit $\vep'$.
			\end{prop}
			\begin{proof}
				Because $T\in P(\G)$, $T=T_{1,\vep}$ for some unit $\vep$ of $\Z[i]$.  Thus, $\overline{x}\in \overline{\vep}x+\G$ by Lemma~\ref{RzvepSOC}.  
				Let $T_j=T_{z_j,\vep_j}\in \OC(x+\G)\setminus \SOC(x+\G)$ for $j\in\set{1,2}$.  If $g\vcentcolon=\gcd(z_1,z_2)$ and $z_j=h_jg$ for $j\in\set{1,2}$, then it follows 
				from Lemma~\ref{RzvepSOC} that 
				\[
					g\big(\vep_2\overline{\vep_1}h_2\overline{h_1}-\overline{h_2}h_1\big)\overline{x}
					=\overline{\vep_1h_1}(\vep_2z_2\overline{x}-\overline{z_2}x)-\overline{\vep_1h_2}(\vep_1z_1\overline{x}-\overline{z_1}x)\in\G.
				\]
				Since $\overline{x}\in\overline{\vep}x+\G$, we have $\big(\vep_2\overline{\vep_1}h_2\overline{h_1}-\overline{h_2}h_1\big)x\in(1/g)\G$.  
				This, together with \eqref{prodrefsq}, implies that
				$\big(\vep_2\overline{\vep_1}h_2\overline{h_1}-\overline{h_2}h_1\big)x\in(1/g)\G\cap(1/\overline{g})\G=\G$, 
				and thus $T_2T_1\in {\OC(x+\G)}$.
				From Lemma~\ref{OCsq}, $\OC(x+\G)$ is a subgroup of $\OC(\G)$.
				
				In addition, any coincidence reflection $T'=T_{z',\vep'}$ of $x+\G$ can be written as $T'=R'\cdot T$ where $R'=R_{z',\overline{\vep}\vep'}\in \SOC(x+\G)$. 
				Hence, $\OC(x+\G)$ is the semidirect product of $\SOC(x+\G)$ and $\langle T\rangle$.
					
				Suppose $\OC(x+\G)$ does not contain any reflection symmetry and $T_{z,\vep}\in \OC(x+\G)$ for some unit $\vep$ of $\Z[i]$.   Since $R\in \SOC(x+\G)$, 
				$R^{-1}\cdot T_{z,\vep}\in \OC(x+\G)$ by Lemma~\ref{OCsq}.  This is a contradiction because $R^{-1}\cdot T_{z,\vep}=T_{1,\overline{\vep'}\vep}\in 
				P(\G)$.
			\end{proof}		

		\subsection{Determination of $\SOC(x+\G)$ and $\OC(x+\G)$}
		
			We now turn to the actual computation of $\OC(x+\G)$ for specific values of $x$. We start with the case when $x$ has an irrational component. Here, the sets $\SOC(x+\G)$
			and $\OC(x+\G)$ are small and thus can be determined completely. The results are summarized in the following theorem, which has been announced in~\cite{LZ10} without proof.
	
			\begin{thm}\label{irrational}  
				Let $\G=\Z[i]$ and $x=a+bi$, with $a,b\in\R$.  If $a$ or $b$ is irrational then $\OC(x+\G)$ is a group of at most two elements.  In particular, if
				\begin{enumerate}[\rm(i)]
					\item $a$ is irrational and $b$ is rational then $\OC(x+\G)=\begin{cases}
					\langle T_r\rangle, &\text{if }2b\in\Z\\
					\set{\mathbbm{1}}, &\text{otherwise}.
					\end{cases}$				
				
					\item $a$ is rational and $b$ is irrational then $\OC(x+\G)=\begin{cases}
					\langle T_{1,-1}\rangle, &\text{if }2a\in\Z\\
					\set{\mathbbm{1}}, &\text{otherwise}.
					\end{cases}$
					
					\item both $a$ and $b$ are irrational, and
					\begin{enumerate}[\rm(a)]
						\item $1$, $a$, and $b$ are rationally independent then $\OC(x+\G)=\set{\mathbbm{1}}$.
					
						\item $a=(p_1/q_1)+(p_2/q_2)b$ where $p_j$, $q_j\in\Z$, and $p_j$ is relatively prime to $q_j$ for $j\in\set{1,2}$, with
						\begin{enumerate}[\rm 1.]
							\item $p_2q_2$ even, then \[\OC(x+\G)=\begin{cases}
							\big\langle T_{p_2+q_2i,1}\big\rangle, &\text{if }q_1\mid 2q_2\\
							\{\mathbbm{1}\}, &\text{otherwise}.
							\end{cases}\]
			
							\item $p_2q_2$ odd, then \[\OC(x+\G)=\begin{cases}
							\big\langle T_{(p_2+q_2)/2-(p_2-q_2)i/2, i}\big\rangle, &\text{if }q_1\mid q_2\\
							\{\mathbbm{1}\}, &\text{otherwise}.
							\end{cases}\]
						\end{enumerate}
					\end{enumerate}
				\end{enumerate}
			\end{thm}
			\begin{proof}
				Suppose either $a$ or $b$ is irrational, that is, $x\notin\Q(i)$.  If $R_{z,\vep}\in \SOC(x+\G)$ then it follows from Lemma~\ref{RzvepSOC} that 
				$\vep z-\overline{z}=0$.  Thus, $\vep z/\overline{z}=1$ which means that $\SOC(x+\G)=\set{\mathbbm{1}}$, where $\mathbbm{1}$ is the identity 
				isometry.
				
				Assume $\OC(x+\G)$ includes two distinct reflections $T_1=T_{h_1g,\vep_1}$ and $T_2=T_{h_2g,\vep_2}$, with $h_1$ and $h_2$ relatively prime.  One obtains 
				from~\eqref{prodrefsq} that $\vep_2\overline{\vep_1}h_2\overline{h_1}-\overline{h_2}{h_1}=0$.  This implies that $T_1=T_2$.  Therefore, either $\OC(x+\G)=\set{\mathbbm{1}}$ 
				or $\OC(x+\G)=\langle T\rangle$ for some coincidence reflection $T$.
				
				Let $T_{z,\vep}\in \OC(\G)\setminus \SOC(\G)$.  If $a$ is irrational and $b$ is rational, then $2\Re{z\overline{x}}$, $\Re{z\overline{x}}\pm\Im{z\overline{x}}\notin\Z$. 
				This means that ${\vep z\overline{x}-\overline{z}x}\in\Z[i]$ if and only if $\vep=z=1$.  If $z=1$, one has $2\Im{\overline{x}}=-2b\in\Z$ and (i) now follows from 
				Lemma~\ref{RzvepSOC}.  The proof of (ii) proceeds analogously.
				
				Suppose now that both $a$ and $b$ are irrational.  From Lemma~\ref{RzvepSOC}, one obtains that a coincidence reflection $T_{z,\vep}\in \OC(x+\G)$ if and only if				
				\begin{equation}\label{refcondsq}
					a=\begin{cases}
					\tfrac{t}{2\Im{z}}+\tfrac{\Re{z}}{\Im{z}}b, &\text{if }\vep=1\\
					\tfrac{t}{2\Re{z}}-\tfrac{\Im{z}}{\Re{z}}b, &\text{if }\vep=-1\\
					\tfrac{t}{\Re{z}+\Im{z}}+\tfrac{\Re{z}-\Im{z}}{\Re{z}+\Im{z}}b, &\text{if }\vep=i\\
					\tfrac{t}{\Re{z}-\Im{z}}-\tfrac{\Re{z}+\Im{z}}{\Re{z}-\Im{z}}b, &\text{if }\vep=-i,
					\end{cases}
				\end{equation}
				for some $t\in\Z$.  In each case, one is able to write $a$ uniquely as $a=c+d\cdot b$ where $c,d\in\Q$.  
				
				Assume that $a=(p_1/q_1)+(p_2/q_2)b$ where $p_j$, $q_j\in\Z$ with $p_j$ and $q_j$ relatively prime for $j\in\set{1,2}$.  If $p_2q_2$ is even 
				then $a$ is expressible in the form \eqref{refcondsq} if and only if $\vep=\pm 1$ and $q_1\mid 2q_2$.  Then, one can simply take $\vep=1$ and 
				$z=p_2+q_2i$ so that $T_{z,\vep}\in \OC(x+\G)$.  The case where $p_2q_2$ is odd is analogous.
			\end{proof}

			\begin{rem}
				Note that for $T_{(p_2+q_2)/2-(p_2-q_2)i/2,i}=R\cdot T_r$ in Theorem~\ref{irrational}, $R$ actually corresponds to multiplication by the 
				complex number $z/\overline{z}$ with $z=p_2+q_2i$.  However, in such a representation, $z$ and $\overline{z}$ are not relatively prime.
			\end{rem}
	
			It only remains to consider the case when both components of $x$ are rational.  Suppose that $x=a+bi\in\Q(i)$ and write $x=p/q$, where $p$,
			$q\in\Z[i]$ with $p$ and $q$ relatively prime.  It turns out that $\SOC(x+\G)$ ultimately depends on the denominator $q$ of $x$.  In particular, we 
			have the following lemma (see~\cite[Lemma 6]{LZ10} or~\cite[Lemma 3.28]{L10}).
			\begin{lemma}\label{divrule} 
				Let $\G=\Z[i]$, $x=p/q$ where $p$, $q\in\Z[i]$ with $p$ and $q$ relatively prime, and $R=R_{z,\vep}\in \SOC(\G)$.  Then
				$R\in \SOC(x+\G)$ if and only if $q$ divides $\vep z-\overline{z}$.  Consequently, $\SOC(x+\G)=\SOC(1/q+\G)$. 
			\end{lemma}
			Hence, in the case of $\SOC(x+\G)$, it is sufficient to restrict the discussion to shifts of the form $x=1/q$. As an immediate consequence of the
			divisibility condition set forth in Lemma~\ref{divrule} we obtain the following results.
			\begin{cor}\label{SOCsub}
				If $q_1,q_2\in\Z[i]$ such that $q_1\mid q_2$, then \[\SOC(\tfrac{1}{q_2}+\G)\subseteq \SOC(\tfrac{1}{q_1}+\G).\]
			\end{cor}
			This implies that the groups $\SOC(x+\G)$ form a lattice in the algebraic sense of a partially ordered set where each pair of elements has a supremum
			and an infimum. It is not difficult to see that the supremum and infimum of $\SOC(1/q_1+\G)$ and $\SOC(1/q_2+\G)$ are given by $\SOC(1/\gcd(q_1,q_2)+\G)$
			and $\SOC(1/\lcm(q_1,q_2) +\G)$, respectively. The latter can be expressed in terms of $\SOC(1/q_1+\G)$ and $\SOC(1/q_2+\G)$.
			\begin{cor}\label{intSOC}
				Suppose $q_1,q_2\in\Z[i]$.  Then
				\[\SOC\big(\tfrac{1}{\lcm(q_1,q_2)}+\G\big)=\SOC(\tfrac{1}{q_1}+\G)\cap \SOC(\tfrac{1}{q_2}+\G).\]				
			\end{cor}
			\begin{proof}
				The forward inclusion follows from Corollary~\ref{SOCsub}.  Suppose that $R=R_{z,\vep}\in \SOC(\G)$ is a coincidence isometry of $1/q_1+\G$ and
				$1/q_2+\G$.  Then, by Lemma~\ref{divrule}, $\lcm(q_1,q_2)$ divides $\vep z-\overline{z}$ and so $R\in \SOC(1/\lcm(q_1,q_2)+\G)$.  
			\end{proof}
			In fact, it is sufficient to consider only shifts of the form $x=1/n$ with $n\in\Z$, as we have the following result.
			\begin{cor}\label{conjSOC}
				If $q\in\Z[i]$ then \[\SOC(\tfrac{1}{q}+\G)=\SOC(\tfrac{1}{\overline{q}}+\G)=\SOC(\tfrac{1}{\lcm{(q,\overline{q})}}+\G).\]
			\end{cor}
			\begin{proof}
				Note that $q\mid(\vep z-\overline{z})$ if and only if $\overline{q}\mid(\vep z-\overline{z})$.  The equalities then follow from Lemma~\ref{divrule} and Corollary~\ref{intSOC}.
			\end{proof}
	
			If $x$ is of the form $x=1/q$ where $q$ is an odd integer, then we get additional information on the elements of $\SOC(1/q+\G)$ as well as their indices.		
			\begin{prop}\label{oddden}
				Let $\G=\Z[i]$ and $q>1$ be an odd rational integer.  If $R=R_{z,\vep}\in \SOC(1/q+\G)$ then the following holds.
				\begin{enumerate}[\rm(i)]
					\item For all other units $\vep'\neq\vep$, $R_{z,\vep'}\notin \SOC(1/q+\G)$. 
					
					\item The coincidence index $\S(R)$ is not divisible by $q$.
				\end{enumerate}
			\end{prop}
			\begin{proof}
				By Lemma~\ref{divrule}, $q\mid (\vep z-\overline{z})$.
				\begin{enumerate}[\rm(i)]
					\item Assume to the contrary that $R_{z,\vep'}\in \SOC(1/q+\G)$ for some unit $\vep'\neq\vep$ of $\Z[i]$.   Then $q\mid (\vep'z-\overline{z})$
					from Lemma~\ref{divrule} which implies that $q$ divides $z$.  However, $q$ is a rational integer, and so $q$ divides both real and imaginary parts of
					$z$.  This is impossible by the choice of $z$.
					
					\item Suppose $q$ divides $\S(R)=N(z)=z\overline{z}$. Since $q$ divides $z(\vep z-\overline{z})$, the rational integer $q$ also divides $z^2$ which yields
					a contradiction.
				\end{enumerate}
			\end{proof}
	
			It is a well-known fact that $\Z[i]$ is a Euclidean domain. That is, for any $a,b\in\Z[i]$ with $b\neq 0$, there exist $k,r\in\Z[i]$ such that $a=kb+r$ and
			$N(r)\leq(1/2)N(b)$ (see for instance,~\cite[Theorem 215]{HW08}).  The next proposition makes use of this fact.
	
			\begin{prop}\label{euclalg}
				Let $q>1$ be an odd rational integer and write $z=kq+r$ where $k, r\in\Z[i]$ and $N(r)<(1/2)N(q)$.  Then $R=R_{z,\vep}\in \SOC(1/q+\G)$ if 
				and only if $\overline{r}=\vep r$, that is, when $r$ and $\overline{r}$ are associates in $\Z[i]$.
			\end{prop}
			\begin{proof}
				Note that $\vep z-\overline{z}=\big(\vep k-\overline{k}\,\big)q+(\vep r-\overline{r})$.  
								
				Suppose $R\in \SOC(1/q+\G)$.  Then by Lemma~\ref{divrule}, $q\mid(\vep r-\overline{r})$.  Since $q$ is odd, $(1-i)q$ still divides $\vep r-\overline{r}$.  Thus, 
				$N((1-i)q)=2N(q)\mid N(\vep r-\overline{r})$.  However, $N(\vep r-\overline{r})\leq 4N(r)<2 N(q)$.  Therefore, $N(\vep r-\overline{r})=0$ and so $\overline{r}=\vep r$.
	
				The converse follows immediately by Lemma~\ref{divrule}.
			\end{proof}
			
			It follows from the prime factorization of Gaussian integers (see for instance, \cite{HW08}) that $r$ and $\overline{r}$ are associates in $\Z[i]$ if and only if $r$ is a
			rational integer multiple of 1, $i$, $1+i$, or $1-i$.  Hence,  for all odd rational integers $q>1$,
			\begin{multline*}
				\SOC\lt(\tfrac{1}{q}+\G\rt)=\set{R_{z,1}\in \SOC(\G):z=kq+r; k\in\Z[i], r\in\Z ,0<r<\tfrac{1}{2}q}
				\cup\\\{R_{z,i}\in \SOC(\G):z=kq+(1+i)r; k\in\Z[i],
				r\in\Z, 0<r<\tfrac{1}{2}q\}.
			\end{multline*}			
			Let $\mathcal{V}$ be the set of \emph{visible} (or \emph{primitive}) points of $\Z[i]$, that is, \[\mathcal{V}=\set{z\in\Z[i]:\gcd(\Re{z},\Im{z})=1}\] and
			let $\mathcal{V}'=\set{z\in \mathcal{V}: (1+i)\,\nmid\, z}$ be the set of visible points that are not divisible by $1+i$. 
			If $f_{1/q+\G}(m)$ denotes the number of CSLs for a given index $m$ of the shifted lattice $1/q+\G$, then the Dirichlet series generating
			function for $f_{1/q+\G}(m)$ is given by
			\begin{alignat}{2}
				\Phi_{1/q+\G}(s)&=\sum_{m=1}^{\infty}\dfrac{f_{1/q+\G}(m)}{m^s}\notag\\
				&=\sum_{\underset{\sst\gcd(r,q)=1}{0<r<\frac{q}{2}}}
				\bigg(\sum_{\underset{\sst kq+r\in\mathcal{V}'}{k\in\Z[i]}}\frac{1}{N(kq+r)^s}+
				\sum_{\underset{\sst kq+(1+i)r\in\mathcal{V}'}{k\in\Z[i]}}\frac{1}{N(kq+(1+i)r)^s}\bigg).\label{sumodd}
			\end{alignat}
			In this case, the generating function cannot be written as an Euler product.  To visualize the set that the sum in \eqref{sumodd} runs over, consider the grid 
			\[L_q=q\Z[i]+\set{r,ir,(1+i)r,(1-i)r:r\in\R}\]
			(see Figure~\ref{visible}).  Observe that the sum is taken over one-fourth of the points of $\mathcal{V'}$ lying on the grid $L$, that is, one point
			out of the four points of $\mathcal{V'}\cap L$ that are equivalent under the action of $C_4$ appears in the sum.
			
			\begin{figure}[ht]
				\begin{center}
					\includegraphics{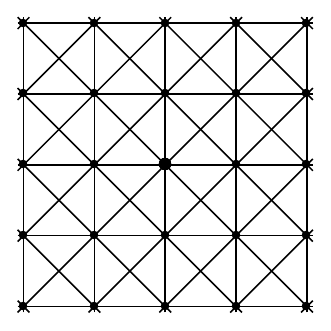}
				\end{center}
				\caption{The grid $L_q$. The black dots are points of $q\Z[i]$.}
				\label{visible}
			\end{figure}
			
			Proposition~\ref{euclalg} also gives the following lower bound on the coincidence index of a coincidence rotation of $1/q+\G$.
					
			\begin{cor}
				Let $\G=\Z[i]$ and $q>1$ be an odd rational integer.  If $R=R_{z,\vep}\in \SOC(1/q+\G)\setminus P(\G)$ then $\S(R)>(1/2)q^2$.
			\end{cor}
			\begin{proof}
				Write $z=kq+r$ where $k,r\in\Z[i]$ and $N(r)<(1/2)N(q)$.  If $\S(R)\leq (1/2)q^2$, then $r=z$.  However, $z/\overline{z}$ is not a unit which contradicts 
				Proposition~\ref{euclalg}.
			\end{proof}

			Finally, we want to return to $\OC(x+\G)$. The picture is far less complete here and we just mention the following result.	
			\begin{prop}\label{denOCgroup}
				Let $x=p/q$ where $p,q\in\Z[i]$ with $p$ and $q$ relatively prime.  If none of the prime factors of $N(q)$ is a splitting prime of $\Z[i]$, then 
				$\OC(x+\G)$ is a group.
			\end{prop}
			\begin{proof}
				From Lemma~\ref{OCsq}, it suffices to show that the product of any two coincidence reflections $T_1=T_{z_1,\vep_1}$ and $T_2=T_{z_2,\vep_2}$ of $x+\G$ is 
				in $\SOC(x+\G)$ to prove the claim.
				
				Since none of the prime factors of $N(q)$ splits in $\Z[i]$, $\overline{q}=\overline{u}q$ for some unit $u$ of $\Z[i]$.  It follows from
				Lemma~\ref{RzvepSOC} that for $j\in\set{1,2}$, $q\mid (u\vep_jz_j\overline{p}-\overline{z_j}p)$.  Set $g\vcentcolon=\gcd(z_1,z_2)$ and write $z_j=h_jg$ for
				$j\in\set{1,2}$.  Then $q$ divides $u\vep_2z_2\overline{z_1}\overline{p}-u\vep_1\overline{z_2}z_1\overline{p}=
				u\vep_1g\overline{g}\overline{p}\big(\vep_2\overline{\vep_1}h_2\overline{h_1}-\overline{h_2}h_1\big)$, and hence, 
				$q\mid\lt(\vep_2\overline{\vep_1}h_2\overline{h_1}-\overline{h_2}h_1\rt)$.  Finally, because $T_2T_1=R_{h_2\overline{h_1},
				\vep_2\overline{\vep_1}}\in \SOC(\G)$, the product $T_2T_1\in \SOC(x+\G)$ by Lemma~\ref{divrule}.
			\end{proof} 			

		\subsection{Specific examples}  
		
			In order to illustrate our results we now explicitly compute $\OC(x+\G)$ for certain values of ${x\in\Q(i)}$. We discuss three examples, all of which are related 
			to the smallest splitting prime $5$ of $\Z[i]$.  The first two, $x=1/5$ and $x=i/(1+2i)$, share the same group $\SOC(x+\G)$ but their sets $\OC(x+\G)$ differ considerably. 
			The third example is $x=(2+i)/6$, where now the numerator instead of the denominator is related to the splitting prime $5$.  This will provide us with the simplest example 
			where the function counting the coincidence rotations is not multiplicative.  Further examples can be found in~\cite{LZ10,L10}.
                        
			In the following, the number of coincidence rotations and CSLs for a given index $m$ of the shifted lattice $x+\G$ shall be denoted by $\hat{f}_{x+\G}(m)$ and $f_{x+\G}(m)$, 
			respectively.

			\begin{ex}\label{den5} 
				Let us consider the case $x=1/5$ first. As $x$ is real, it is invariant under complex conjugation, or in other words, there exists a reflection 
				leaving $x+\G$ invariant.  This assures us that $\OC(x+\G)$ is a group.
		
				Here, the denominator is $q=5$. Write $z=5k+r$ where $k,r\in\Z[i]$ and $N(r)<25/2$.  For all possible remainders $r$, $\overline{r}$ is not an 
				associate of $r$ if and only if ${5\mid N(z)}$.  It follows then from Propositions~\ref{oddden} and~\ref{euclalg} that for all numerators $z$, there is a (unique) unit 
				$\vep$ of $\Z[i]$ for which $R_{z,\vep}\in \SOC(x+\G)$ if and only if $5\nmid N(z)$.  This means that $\SOC(x+\G)\cong{\Z}^{(\aleph_0)}$.  Moreover, 
				$\OC(x+\G)=\SOC(x+\G)\rtimes\langle T_{r}\rangle$ by Proposition~\ref{OCsqprop}, and
				\[\hat{f}_{x+\G}(m)=f_{x+\G}(m)=\begin{cases} f_{\Z^2}(m), &\text{if }5\nmid m\\ 0, &\text{otherwise.}\end{cases}\]  
				The function $f_{x+\G}$ is still multiplicative and the Dirichlet series generating function for $f_{x+\G}(m)$ is given by
				\begin{alignat*}{2}
					\Phi_{x+\G}(s)&=\sum_{m=1}^{\infty}{\frac{f_{x+\G}(m)}{m^s}}=\frac{1-5^{-s}}{1+5^{-s}}\cdot\Phi_{\Z^2}(s)\\
									&=1+\tfrac{2}{13^s}+\tfrac{2}{17^s}+\tfrac{2}{29^s}+\tfrac{2}{37^s}+\tfrac{2}{41^s}+\tfrac{2}{53^s}+\tfrac{2}{61^s}+\tfrac{2}{73^s}+\\
									&\quad{}\tfrac{2}{89^s}+\tfrac{2}{97^s}+\tfrac{2}{101^s}+\tfrac{2}{109^s}+\tfrac{2}{113^s}+\tfrac{2}{137^s}+\tfrac{2}{149^s}+\tfrac{2}{157^s}+\\
									&\quad{}\tfrac{2}{169^s}+\tfrac{2}{173^s}+\tfrac{2}{181^s}+\tfrac{2}{193^s}+\tfrac{2}{197^s}+\tfrac{4}{221^s}+\tfrac{2}{229^s}+\cdots\,.
				\end{alignat*}
				One can show using a specific case of Delange's theorem (see for instance, \cite[Theorem 5 of Appendix]{BM99}) that the number of CSLs of $x+\G$ with index
				at most $N$ is asymptotically $2N/(3\pi)$.
			\end{ex}
			
			\begin{ex}\label{exnotgroup}
				Setting $x=i/(1+2i)$ provides us with an example where $\OC(x+\G)$ is not a group.  In this case, the denominator of $x$ is $q=1+2i$.  Since $5=\lcm(q,\overline{q})$, 
				\[\SOC(x+\G)=\SOC(\tfrac{1}{5}+\G)\cong{\Z}^{(\aleph_0)}\] by 
				Corollary~\ref{conjSOC}.  Observe that $\OC(x+\G)$ does not include a reflection symmetry by Lemma~\ref{RzvepSOC}.  From Example~\ref{den5} and Proposition~\ref{OCsqprop}, 
				we have $5\mid N(z)$ whenever $T_{z,\vep}=R_{z,\vep}\cdot T_r\in \OC(x+\G)$.  

				Given a numerator $z$ whose norm is divisible by 5, either $1+2i$ or $1-2i$ (and not both) appears in the factorization of $z$ into primes of $\G$.  If 
				$(1-2i)\mid z$, then $z\overline{x}\in\G$ which means that ${\vep z\overline{x}-\overline{z}x\in\G}$ for all units $\vep$ of $\Z[i]$.  On the other hand,
				if $(1+2i)\mid z$ then $\vep z\overline{x}-\overline{z}x=i(-\vep y-\overline{y})/5$, where $y=(1+2i)z$.  This implies that 
				$\vep z\overline{x}-\overline{z}x\notin\G$ for all units $\vep$ of $\Z[i]$, since otherwise, $R_{y,-\vep}\in \SOC(1/5+\G)$ by Lemma~\ref{divrule} 
				which is impossible because $5\mid N(y)$.  
				
				Therefore, by Lemma~\ref{RzvepSOC},
				\[\OC(x+\G)=\SOC(x+\G)\cup\set{T_{z,\vep}\in \OC(\G): (1-2i)\mid z}.\]
				We claim that $\OC(x+\G)$ is not a group.  Indeed, if $T_j=T_{z,\vep_j}\in \OC(x+\G)\setminus \SOC(x+\G)$ for $j\in\set{1,2}$ with $\vep_1\neq\vep_2$,  	
				then $T_2T_1\notin \SOC(x+\G)$.
		
				Since $\SOC(x+\G)=\SOC(1/5+\G)$, one concludes that $\hat{f}_{x+\G}(m)=\hat{f}_{1/5+\G}(m)$.  Denote by $\hat{F}_{x+\G}(m)$ the number of 
				linear coincidence isometries of $x+\G$ of index $m$.  Since each non-identity rotation symmetry is not a coincidence rotation of $x+\G$, by 
				Proposition~\ref{countCSLshift}, $f_{x+\G}(m)=\hat{F}_{x+\G}(m)$.  It is remarkable that $f_{x+\G}$ is still multiplicative, even though $\OC(x+\G)$ is not a group.  
				It is given by 
				\[f_{x+\G}(p^r)=\begin{cases}
				2, &\text{if }p\equiv 1\imod{4}\text{ and }p\neq 5\\
				4, &\text{if }p=5\\
				0, &\text{otherwise},
				\end{cases}\]
				for primes $p$ and $r\in\N$.  The Dirichlet series generating function for $f_{x+\G}(m)$ reads
				\begin{alignat*}{2}
					\Phi_{x+\G}(s)&=\sum_{m=1}^{\infty}{\frac{f_{x+\G}(m)}{m^s}}=\frac{1+3\cdot 5^{-s}}{1+5^{-s}}\cdot\Phi_{\Z^2}(s)\\
					&=1+\tfrac{4}{5^s}+\tfrac{2}{13^s}+\tfrac{2}{17^s}+\tfrac{4}{25^s}+\tfrac{2}{29^s}+\tfrac{2}{37^s}
					+\tfrac{2}{41^s}+\tfrac{2}{53^s}+\tfrac{2}{61^s}+\tfrac{8}{65^s}+\tfrac{2}{73^s}+\cdots\,.					
				\end{alignat*}
				Looking at $\Phi_{x+\G}(s)$, we have that the number of CSLs of $x+\G$ of index at most $N$ is asymptotically given by $4N/(3\pi)$.		
			\end{ex}	
			
			\begin{ex} 
				Our last example is $x=(2+i)/6$. Here,  the denominator of $x$ is $q=6=2\cdot 3$.  Hence, by Corollary~\ref{intSOC}, 
				\[\SOC(x+\G)=\SOC(\tfrac{1}{2}+\G)\cap \SOC(\tfrac{1}{3}+\G).\] 
				From~\cite[Example 3]{LZ10},  $R_{z,\vep}\in \SOC(1/2+\G)$ if and only if $\vep=\pm 1$.  Write $z=3k+r$, where $k,r\in\Z[i]$ and $N(r)<9/2$. 
				Note that for all possible remainders $r$, $\vep=\overline{r}/{r}=\pm 1$ if and only if $N(r)\equiv {1\imod{3}}$. It follows then from Proposition~\ref{euclalg} that 
				$R_{z,\vep}\in \SOC(x+\G)$ for some (unique) $\vep\in\set{1,-1}$ if and only if $N(z)\equiv{1\imod{3}}$.
				Thus, ${\SOC(x+\G)}\cong{\Z}^{(\aleph_0)}$ and	
				\[\hat{f}_{x+\G}(m)=\begin{cases}
					f_{\Z^2}(m), &\text{if }m\equiv 1\imod{3}\\
					0, &\text{otherwise}.
				\end{cases}\]
				Here, $\hat{f}_{x+\G}$ is not multiplicative anymore despite the fact that both $\hat{f}_{1/2+\G}$ and $\hat{f}_{1/3+\G}$ are multiplicative~\cite{LZ10,L10}.
				However, $\hat{f}_{x+\G}(m)=(1/2)\big(1+\chi_{-3}(m)\big)f_{\Z^2}(m)$ is the sum of two multiplicative functions.  Hence, each term of $\hat{f}_{x+\G}(m)$ has 
				an Euler product which allows us to explicitly calculate its Dirichlet series generating function given by
				\begin{alignat*}{2}
					\hat{\Phi}_{x+\G}(s)&=\sum_{m=1}^{\infty}{\frac{\hat{f}_{x+\G}(m)}{m^s}}\\
					&=\tfrac{1}{2}\Phi_{\Z^2}(s)+\frac{1}{1-2^{-s}}\cdot\frac{1}{1-3^{-2s}}\cdot\frac{L(s,\chi_{-3})L(s,\chi_{12})}{2\zeta(2s)}\\
					&=1+\tfrac{2}{13^s}+\tfrac{2}{25^s}+\tfrac{2}{37^s}+\tfrac{2}{61^s}+\tfrac{2}{73^s}+\tfrac{4}{85^s}+\tfrac{2}{97^s}+\cdots\,,
				\end{alignat*}
				where $L(s,\chi_{-3})$ and $L(s,\chi_{12})$ are the $L$-series of the primitive Dirichlet characters 
				\begin{alignat*}{2}
					\chi_{-3}(m)&=\begin{cases}
					1, & \text{if }m\equiv 1\imod{3}\\
					-1, & \text{if }m\equiv 2\imod{3}\\
					0, & \text{otherwise}\end{cases}
					\quad\text{and}\\
					\chi_{12}(m)&=\begin{cases}
					1, & \text{if }m\equiv 1,11\imod{12}\\
					-1, & \text{if }m\equiv 5,7\imod{12}\\
					0, & \text{otherwise},
					\end{cases}
				\end{alignat*}
        respectively.  One obtains that the number of coincidence rotations of $x+\G$ of index at most $N$ is asymptotically $N/(2\pi)$.
				
				Again, ${\OC(x+\G)}$ does not contain a reflection symmetry. Nevertheless, $\OC(x+\G)$ forms a group by Proposition~\ref{denOCgroup} since $N(q)=2^2\cdot 3^2$.
				Proposition~\ref{OCsqprop} indicates that if the coincidence reflection $T_{z,\vep}\in \OC(x+\G)$ then $N(z)\equiv {2\imod{3}}$. Conversely, suppose that  $z$ is a numerator
				with $N(z)\equiv {2\imod{3}}$.  Observe that the numerator $p=2+i$ of the shift $x$ is a factor of $5$ which splits in $\Z[i]$.  This means that if $p\nmid z$, 
				$y\vcentcolon=z\overline{p}$ is still a numerator corresponding to some coincidence rotation of $\G$.  In fact, because ${N(y)\equiv  1\imod{3}}$, $R_{y,\vep}\in \SOC(x+\G)$
				for some (unique) $\vep\in\set{1,-1}$.  Hence, $6\mid(\vep y-\overline{y})$ by Lemma~\ref{divrule}, and one obtains that 
				$\vep z\overline{x}-\overline{z}{x}=(1/6)(\vep y-\overline{y})\in\G$.  The case where $p\mid z$ yields the same result.  Thus, $T_{z,\vep}\in \OC(x+\G)$ by 
				Lemma~\ref{RzvepSOC}.  Altogether one has
				\[
					\OC(x+\G)=\SOC(x+\G)\;\cup
					\set{T_{z,\vep}: N(z)\equiv 2\imod{3} 
					\text{ and }\vep=\begin{cases}
								\sst{1,}&\sst{\text{if }3\nmid\Re{z\overline{p}}}\\
								\sst{-1,}&\sst{\text{if }3\mid\Re{z\overline{p}}}
							\end{cases}}.
				\]
				From this we infer $\hat{F}_{x+\G}(m)=f_{x+\G}(m)=f_{\Z^2}(m)$, where $\hat{F}_{x+\G}(m)$ counts the number of linear coincidence isometries of $x+\G$ of a given index $m$. 
				Note that $f_{x+\G}$ and $\hat{F}_{x+\G}$ are multiplicative, whereas $\hat{f}_{x+\G}$ is not.
			\end{ex}

	\section{Linear coincidences of crystallographic point packings}\label{sectmultlatt}

		We now take a further step and consider the coincidence problem this time for sets of points formed by finite unions of shifted lattices.  Such sets are of 
		particular interest in crystallography because they are a standard model for ideal crystals. We briefly recall the notion of crystallographic point packings here and refer for 
		further reading to~\cite{BG13,PZ98} and references therein.

		A subset $L$ of $\Rd$ shall be called a \emph{crystallographic point packing} or a \emph{multilattice generated by the lattice $\G$ in $\Rd$} if $L$ is the union of $\G$ and a
		finite number of translated copies of $\G$, that is, $L=\bigcup_{k=0}^{m-1}(x_k+\G)$ where $x_k\in\Rd$, $m\in\N$, and $x_0=0$.  In general, a crystallographic point packing is not a 
		lattice.  An orthogonal transformation $R\in \OG(d)$  will be called a \emph{linear coincidence isometry of $L$} if $L(R)\vcentcolon=L\cap RL$ includes a cosublattice of 
		some shifted lattice $x_k+\G$, $0\leq k\leq m-1$.  The intersection $L(R)$ shall be referred to as the \emph{coincidence site packing} (CSP) of $L$ 
		generated by $R$.  The density of $L(R)$ in $L$, by this we mean the ratio of the density of points in $L$ by the density of points in $L(R)$, is the 
		\emph{coincidence index of $R$ with respect to $L$}, which is denoted by $\S_L(R)$.  Note that $\S_L(R)$ is not necessarily an integer.

		The next lemma describes exactly when the intersection of the shifted lattice $x_k+\G$ and the image of the shifted lattice $x_j+\G$ under a linear isometry
		forms a cosublattice of $x_k+\G$.
		
		\begin{lemma}\label{intersectLgen}
			Suppose $\G$ is a lattice in $\Rd$, $R\in \OG(d)$, and $x_j, x_k\in\Rd$.  Then $(x_k+\G)\cap R(x_j+\G)$ contains a cosublattice of $x_k+\G$ if and only if 
			$R\in \OC(\G)$ and $Rx_j-x_k\in\G+R\G$.  Moreover, if $Rx_j-x_k\in \ell_{j,k}+R\G$  with $\ell_{j,k}\in\G$, then 
			\begin{equation}\label{CSML}
					(x_k+\G)\cap R(x_j+\G)=(x_k+\ell_{j,k})+\G(R).
			\end{equation}
		\end{lemma}
		\begin{proof}
			Write $(x_k+\G)\cap R(x_j+\G)=(x_k,\mathbbm{1}_d)[\G\cap(Rx_j-x_k,R)\G]$.  Then the intersection $(x_k+\G)\cap R(x_j+\G)$ contains a cosublattice of 
			$x_k+\G$ if and only if $R\in \OC(\G)$ and $Rx_j-x_k\in\G+R\G$ by Theorem~\ref{affisom}.  Equation~\eqref{CSML} follows from~\eqref{ACSL}.
		\end{proof}
		
		Equation~\eqref{CSML} tells us that given an $R\in \OC(\G)$ satisfying $Rx_j-x_k\in\G+R\G$, then the intersection $(x_k+\G)\cap R(x_j+\G)$ does not only 
		contain a cosublattice of $x_k+\G$, but is itself a cosublattice of $x_k+\G$.  In addition, the index of the cosublattice $(x_k+\G)\cap R(x_j+\G)$ in $x_k+\G$ 
		is $\S(R)$.
		
		\begin{rem}
			Let $\G\subseteq\Rd$ be a lattice, $R\in \OG(d)$, and $x_j, x_k\in\Rd$.  The intersection  $(x_k+\G)\cap R(x_j+\G)$ is a cosublattice of $x_k+\G$ if and 
			only if it is a cosublattice of $Rx_j+R\G$.  Indeed, if $Rx_j-x_k\in Rt_{j,k}+\G$ with $t_{j,k}\in\G$ then 
			\begin{equation}\label{CSML2}
				(x_k+\G)\cap R(x_j+\G)=(Rx_j-Rt_{j,k})+\G(R).
			\end{equation}
			The cosublattice $(x_k+\G)\cap R(x_j+\G)$ is also of index $\S(R)$ in $Rx_j+R\G$.
		\end{rem}
		
		The following theorem gives the solution of the coincidence problem for a crystallographic point packing.

		\begin{thm}\label{genthm}
			Let $L=\bigcup_{k=0}^{m-1}(x_k+\G)$ be a crystallographic point packing generated by the lattice $\G$ in $\Rd$, where $x_k\in\Rd$ for $0\leq k\leq m-1$, 
			$x_0=0$, and $x_k-x_j\notin\G$ whenever $k\neq j$.
			\begin{enumerate}[\rm(i)]
				\item The set of linear coincidence isometries of $L$ is $\OC(\G)$.

				\item Given an $R\in \OC(\G)$, let \[\sigma=\set{(x_j,x_k):Rx_j-x_k\in\G+R\G}.\]  Then \[\S_L(R)=\tfrac{m}{\abs{\sigma}}\S(R).\]  In addition, if 
				$Rx_j-x_k=\ell_{j,k}+Rt_{j,k}$ with $\ell_{j,k}, t_{j,k}\in\G$, then 
				\begin{equation}\label{unionform}
					L(R)=\bigcup_{(x_j,x_k)\in \sigma}[(x_k+\ell_{j,k})+\G(R)]
						=\bigcup_{(x_j,x_k)\in \sigma}[(Rx_j-Rt_{j,k})+\G(R)].					
				\end{equation}
			\end{enumerate}
		\end{thm}
		\begin{proof}
			The intersection $L(R)$ can be expressed as the disjoint union 
			\begin{equation}\label{expand}
				L(R)=L\cap RL=\bigcup_{j=0}^{m-1}\bigcup_{k=0}^{m-1}[(x_k+\G)\cap R(x_j+\G)].
			\end{equation}
    	\begin{enumerate}[\rm(i)]
				\item Suppose $R$ is a linear coincidence isometry of $L$.  Then there is some shifted lattice $x_k+\G$ for which $(x_k+\G)\cap RL$ contains a cosublattice
				of $x_k+\G$.  Thus, ${(x_k+\G)\cap R(x_j+\G)\neq\varnothing}$ for some $j$ with $0\leq j\leq m-1$.  However, the number of shifted copies of $\G$ in $L$ 
				is finite.  This implies that the intersection $(x_k+\G)\cap R(x_j+\G)$ must be also a cosublattice of $x_k+\G$.  It now follows from 
				Lemma~\ref{intersectLgen} that $R\in \OC(\G)$.  Conversely, if $R\in \OC(\G)$ then the sublattice $\G(R)$ of $\G$ appears in $L(R)$.  Thus, $R$ is a 
				linear coincidence isometry of $L$.  
							
				\item Since $(x_0,x_0)\in \sigma$, $\abs{\sigma}\neq 0$.  One sees from Lemma~\ref{intersectLgen} that $(x_k+\G)\cap R(x_j+\G)\neq\varnothing$ whenever
				$(x_j,x_k)\in\sigma$.  Applying~\eqref{CSML} and~\eqref{CSML2} to each intersection of the disjoint union in~\eqref{expand} yields~\eqref{unionform}.
				Now, each $(x_j,x_k)\in\sigma$ contributes a different shifted copy of $\G(R)$ to $L(R)$.  This means that $L(R)$ is made up of $\abs{\sigma}$ distinct 
				shifted copies of $\G(R)$, each of which is of index $\S(R)$ in the respective shifted copy of $\G$ (or $R\G$).  Because $L$ consists of $m$ separate
				shifted copies of $\G$, the formula for $\S_L(R)$ follows.\qedhere
			\end{enumerate}
		\end{proof}
		
		Therefore, the set of linear coincidence isometries of the crystallographic point packing $L$ generated by $\G$ is still $\OC(\G)$, albeit the coincidence indices of an 
		$R\in \OC(\G)$ with respect to $\G$ and $L$ are not necessarily equal.  Moreover, $L(R)$ consists of cosublattices of shifted lattices in~$L$, one of 
		which must always be $\G(R)$.

	\section{Linear coincidences of the diamond packing}\label{diamondpacking}
	
		The \emph{diamond packing} or \emph{tetrahedral packing} \cite{CS99} is made up of two face-centered cubic (f.c.c.) lattices, wherein one of the f.c.c.~lattices is a 
		translate of the other by $(1/4)(a,a,a)$, with $a$ being the length of the edges of a conventional unit cell of the f.c.c.~lattice (see 
		Figure~\ref{diamond}).  It is also known as the packing $D_3^+$ and is not a lattice.  An equivalent way of constructing the diamond packing as a 
		motif of vertices of tetrahedrons and their barycenters can be found in~\cite{S08}.  Here, we use the results of Section~\ref{sectmultlatt} to identify the 
		linear coincidence isometries, coincidence indices, and the resulting intersections of the diamond packing. To this end, we first recall the corresponding 
		results for cubic lattices.
		
		\begin{figure}[ht]
			\begin{center}
				\includegraphics{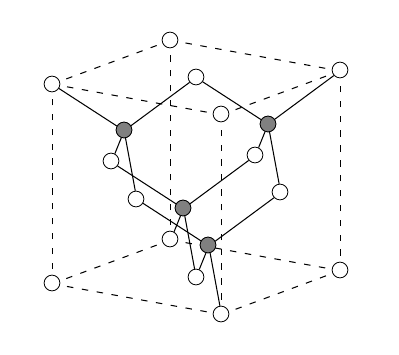}
			\end{center}
			\caption{A conventional unit cell of the diamond packing.  The white dots are part of the f.c.c.~lattice while the gray dots belong to the shifted 
			f.c.c. lattice}\label{diamond}
		\end{figure}

		\subsection{Solution of the coincidence problem for the cubic lattices}
			We see from Section~\ref{sectmultlatt} that it is imperative that we familiarize ourselves with the coincidences of the f.c.c.~lattice (see 
			\cite{G74,GBW74,G84,B97,Z05}) before we even consider the coincidences of the diamond packing.  Let $\G_P=\Z^3$, 
			$\G_B=\G_P\cup[(1/2,1/2,1/2)+\G_P]$, and $\G_F=\G_B^{\ast}$ denote the primitive cubic (p.c.), body-centered cubic (b.c.c.), and 
			f.c.c.~lattice, respectively.  Then $\OC(\G_P)=\OC(\G_B)=\OC(\G_F)=\OG(3,\Q)$, and if $R\in \OG(3,\Q)$, then $\S_{\G_P}(R)=\S_{\G_B}(R)=\S_{\G_F}(R)$
			~\cite{GBW74,B97}.  Therefore, it is enough to look at the coincidences of the p.c.~lattice.
			
			As in the planar case, the analysis of $\OC(\Z^3)$ starts with the group of coincidence rotations of $\Z^3$.  To this end, 
			Cayley's parametrization of matrices in $\SO(3)$ by quaternions is used~\cite{B97}.  Let us first recall some results about quaternions and introduce some notations.  
			Extensive treatments on quaternions can be found in \cite{KR91,CS03,H19,HW08}.
	
			Let $\set{\e,\ii,\j,\k}$ be the standard basis of $\R^4$ where $\e={(1,0,0,0)}^T$, $\ii={(0,1,0,0)}^T$, ${\j={(0,0,1,0)}^T}$, and $\k={(0,0,0,1)}^T$.  The 
			\emph{quaternion algebra over $\R$} is the associative division algebra $\H\vcentcolon=\H(\R)=\R\e+\R\ii+\R\j+\R\k\cong\R^4$ where multiplication is defined by the 
			relations $\ii^2=\j^2=\k^2=\ii\j\k=-\e$.  An element of $\H$ is called a quaternion, and is written as either $q=q_0\e+q_1\ii+q_2\j+q_3\k$ or $q=(q_0,q_1,q_2,q_3)$.  Given 
			two quaternions $q$ and $p$, their \emph{inner product} is defined as the standard scalar	product of $q$ and $p$ as vectors in $\R^4$.  The \emph{conjugate} of a quaternion 
			$q=(q_0,q_1,q_2,q_3)$ is $\overline{q}=(q_0,-q_1,-q_2,-q_3)$, and its \emph{norm} is $\nrq=q\,\overline{q}=q_0^2+q_1^2+q_2^2+q_3^2\in\R$.  
		
			A quaternion whose components are all integers is called a \emph{Lipschitz quaternion}.  On the other hand, a \emph{Hurwitz quaternion} is a quaternion whose components are all
			integers or all half-integers.  The set of Lipschitz quaternions and Hurwitz quaternions shall be denoted by $\L$ and $\J$, respectively.  A \emph{primitive quaternion} $q$ is a
			quaternion in $\L$ whose components are relatively prime. 
		
			Given a quaternion $q=(q_0,q_1,q_2,q_3)$, its \emph{real part} and \emph{imaginary part} are defined as $\Re{q}=q_0$ and $\Im{q}=q_1\ii+q_2\j+q_3\k$,
			respectively.  The \emph{imaginary space of $\H$} is the three-dimensional vector subspace $\Im\H=\set{\Im{q}:q\in\H}\cong\R^3$ of $\H$.  

			An $R\in \SOC(\Z^3)=\SO(3,\Q)$ can be parametrized by a primitive quaternion $q$ so that for all $x\in\R^3$ viewed as an element of $\Im{\H}$, $R(x)=qxq^{-1}$.  
			In such a case, we denote $R$ by $R_q$.  The coincidence index of $R_q\in \SOC(\Z^3)$ is equal to the odd part of $\nrq$, that is, $\S(R_q)=\nrq/2^{\ell}$, where $\ell$ is the 
			largest power of $2$ that divides $\nrq$~\cite{GBW74,G84,B97}.

			Similarly, a primitive quaternion $q$ can be associated to every $T\in \OC(\Z^3)\setminus \SOC(\Z^3)$ so that $T(x)=-qxq^{-1}=q\bar{x}q^{-1}$ for all $x\in\Im\H$, 
			in which case, $T$ shall be written as $T_q$.  The CSLs generated by $T_q$ and $R_q$ are the same, and so $\S(T_q)=\S(R_q)$.

			Let $f_{\Z^3}(m)$ be the number of CSLs of $\Z^3$ of index $m$.  Once again, $f_{\Z^3}$ is multiplicative~\cite{G84,B97} and its Dirichlet series generating function is given by
			\begin{equation}\label{dircubic}
				\begin{aligned}
					\Phi_{\Z^3}(s)&=\sum_{m=1}^{\infty}{\frac{f_{\Z^3}(m)}{m^s}}=\prod_{p\neq 2}{\frac{1+p^{-s}}{1-p^{1-s}}}=
					\frac{1}{1+2^{-s}}\cdot\frac{\zeta_{\J}(s/2)}{\zeta(2s)}\\
					&=1+\tfrac{4}{3^s}+\tfrac{6}{5^s}+\tfrac{8}{7^s}+\tfrac{12}{9^s}+\tfrac{12}{11^s}+\tfrac{14}{13^s}+
					\tfrac{24}{15^s}+\tfrac{18}{17^s}+
					\tfrac{20}{19^s}+\tfrac{32}{21^s}+\tfrac{24}{23^s}+\cdots,
				\end{aligned}
			\end{equation}
			where $\zeta_{\J}(s)=\big(1-2^{1-2s}\big)\zeta(2s)\zeta(2s-1)$ is the zeta function of $\J$ or the Dirichlet series generating function for the number of
			nonzero right ideals of $\J$, compare~\cite{V80}.  Here, the number of CSLs of index at most $N$ is asymptotically given by $3N^2/\pi^2$.  The number of coincidence rotations of 
			$\Z^3$ for a given index $m$ is given by $\hat{f}_{\Z^3}(m)=24f_{\Z^3}(m)$.  Consequently, the Dirichlet series generating function for $\hat{f}_{\Z^3}(m)$ is 
			$24\Phi_{\Z^3}(s)$ (cf.~\cite[Eq.~(3)]{RMRK04}).

 		\subsection{The diamond packing}\label{cubic}
    
			Take $\G$ to be an f.c.c.~lattice.  We identify $\R^3$ with $\Im{\H}$, and associate $\G$ with  
			\[
				\G=2\Im{\L}\cup[(1,1,0)+2\Im{\L}]
				\cup[(0,1,1)+2\Im{\L}]\cup[(1,0,1)+2\Im{\L}].
			\]
			The dual lattice of $\G$ is the b.c.c.~lattice $\G^{\ast}=\Im{\J}$, and the diamond packing is identified with $D_3^+=\G\cup(x+\G)$, where 
			${x=(1/2)(1,1,1)}$.  It follows from Theorem~\ref{genthm} that the group of linear coincidence isometries of $D_3^+$ is $\OC(\G)=\OC(\G^{\ast})$.
	  
			Theorem~\ref{genthm} suggests that it is necessary that we compute for $\OC(x+\G)$ to ascertain the coincidence index of a linear coincidence isometry $R$ of
			$D_3^+$.  To this end, note that $\G+R\G={[\G^{\ast}(R)]}^{\ast}$, that is, $\G+R\G$ is the dual lattice of the CSL $\G^{\ast}(R)$ of $\G^{\ast}$.  The next lemma, stated 
			in~\cite{Z05}, gives a spanning set for $\G^{\ast}(R)$ over $\Z$.

			\begin{lemma}\label{spanset}
				Let $\G^{\ast}=\Im{\J}$ and $R=R_q\in \SOC(\G^{\ast})$ where $q=(q_0,q_1,q_2,q_3)$ is a primitive quaternion. Let 
				\begin{equation}\label{spanvect}
					\begin{aligned}
						\mathbf{r_0}&\vcentcolon=\Im{q}=(q_1,q_2,q_3),\\
						\mathbf{r_1}&\vcentcolon=\Im{q\ii}=(q_0,q_3,-q_2),\\ 
						\mathbf{r_2}&\vcentcolon=\Im{q\j}=(-q_3,q_0,q_1),\\
						\mathbf{r_3}&\vcentcolon=\Im{q\k}=(q_2,-q_1,q_0). 
					\end{aligned}
				  \end{equation}
				Then the CSL $\G^{\ast}(R)$ of $\G^{\ast}$ is the $\Z$-span of the following 
				vectors:
				\begin{enumerate}[\rm(i)]
					\item $\mathbf{r_0}$, $\mathbf{r_1}$, $\mathbf{r_2}$, $\mathbf{r_3}$, $(1/2)(\mathbf{r_0}+\mathbf{r_1}+\mathbf{r_2}+\mathbf{r_3})$ if $\nrq$ is
					odd,
			
					\item $\mathbf{r_0}$, $(1/2)(\mathbf{r_0}+\mathbf{r_1})$, $(1/2)(\mathbf{r_0}+\mathbf{r_2})$, $(1/2)(\mathbf{r_0}+\mathbf{r_3})$ if 
					$\nrq\equiv 2\imod{4}$,
			
					\item $(1/2)\mathbf{r_0}$, $(1/2)\mathbf{r_1}$, $(1/2)\mathbf{r_2}$, $(1/2)\mathbf{r_3}$ if $\nrq\equiv 0\imod{4}$.
				\end{enumerate}
			\end{lemma}
 
			We now proceed to determine $\OC(x+\G)$.  In the succeeding calculations, we embed $\Im\H$ in $\H$ via the canonical projection so that vectors in $\Im\H$ 
			are treated as quaternions whose real part is $0$.  
			
			Observe that for $u\in\set{\e,\ii,\j,\k}$, $R=R_q\in \SOC(\G)$, and $x\in\Im\H$, \[\inn{Rx-x}{\Im{qu}}=\inn{uq-qu}{x}.\]
			Denote by $\times$ the usual vector (cross) product of two vectors in $\Im\H\cong\R^3$.  Given $a,b,c\in\Im{\H}$, one has $a\times b=(1/2)(ab-ba)$ and 
			$\inn{a\times b}{c}=\inn{a}{b\times c}$ (see for instance,~\cite{KR91}).  Together, they imply that $\inn{Rx-x}{\Im{qu}}=-2\inn{q}{u\times x}$ whenever 
			$u\in\set{\ii,\j,\k}$.
			Therefore, substituting the vectors in \eqref{spanvect} yields
			\begin{equation}\label{innprodsv}
				\begin{aligned}
					  \inn{Rx-x}{\mathbf{r_0}}&=0,\\
					  \inn{Rx-x}{\mathbf{r_1}}&=-2\inn{q}{\ii\times x},\\ 
						\inn{Rx-x}{\mathbf{r_2}}&=-2\inn{q}{\j\times x},\\
						\inn{Rx-x}{\mathbf{r_3}}&=-2\inn{q}{\k\times x}.
				\end{aligned}
			\end{equation}

			From now on, let $x=(1/2)(0,1,1,1)$.  Keeping in mind that $Rx-x\in\G+R\G$ if and only if $\inn{Rx-x}{\mathbf{t}}\in\Z$ for all 
			$\mathbf{t}\in\G^{\ast}(R)$, we consider the following three possibilities:

			\noindent\texttt{Case I:} $\nrq$ is odd
			
				By Lemma~\ref{spanset}, $\mathbf{t}=a\mathbf{r_0}+b\mathbf{r_1}+c\mathbf{r_2}+d\mathbf{r_3}+(1/2)e(\mathbf{r_0}+\mathbf{r_1}
				+\mathbf{r_2}+\mathbf{r_3})$, for some $a,b,c,d,e\in\Z$.  It follows from~\eqref{innprodsv} that
				\[\inn{Rx-x}{\mathbf{t}}=-\inn{q}{(0,b,c,d)\times (0,1,1,1)}\in\Z\]
				for all $a,b,c,d,e\in\Z$.  Thus, by Theorem~\ref{OCxG}, $R_q\in \SOC(x+\G)$ whenever $\nrq$ is odd.
		
			\noindent\texttt{Case II:} $\nrq\equiv 2\imod{4}$

				Write $\mathbf{t}=a\mathbf{r_0}+(1/2)b(\mathbf{r_0}+\mathbf{r_1})+(1/2)c(\mathbf{r_0}+\mathbf{r_2})+(1/2)d(\mathbf{r_0}+\mathbf{r_3})$, 
				for some $a,b,c,d\in\Z$, and $q=r+2s$ for some $s\in \J$ and $r\in\set{(1,1,0,0),(1,0,1,0),(1,0,0,1)}$.
				Then
				\[
					\inn{Rx-x}{\mathbf{t}}=-\tfrac{1}{2}\inn{r}{(0,b,c,d)\times (0,1,1,1)}-
					\inn{s}{(0,b,c,d)\times (0,1,1,1)}\notin\Z
				\]
				for some values of $b,c,d\in\Z$.  This means that $R_q\notin \SOC(x+\G)$ if $\nrq\equiv 2\imod{4}$.
		
			\noindent\texttt{Case III:} $\nrq\equiv 0\imod{4}$

				One can express $\mathbf{t}$ as $\mathbf{t}=(1/2)a\mathbf{r_0}+(1/2)b\mathbf{r_1}+(1/2)c\mathbf{r_2}+(1/2)d\mathbf{r_3}$ for some 
				$a,b,c,d\in\Z$.  Write $q=r+2s$ where $s\in\L$ and $r=(1,1,1,1)$.  This yields
				\[\inn{Rx-x}{\mathbf{t}}=-\inn{s}{(0,b,c,d)\times(0,1,1,1)}\in\Z,\]
				for all $a,b,c,d\in\Z$.  Consequently, $R_q\in \SOC(x+\G)$ whenever $\nrq\equiv 0\imod{4}$.
		
			The following lemma summarizes the results for $\OC(x+\G)$.

			\begin{lemma}\label{OCshiftedfcc}
				Let $\G$ be the f.c.c.~lattice 
				\[
					\G=2\Im{\L}\cup[(1,1,0)+2\Im{\L}]
					\cup[(0,1,1)+2\Im{\L}]\cup[(1,0,1)+2\Im{\L}],
				\]
				and $x=(1/2)(1,1,1)$.  Then 
				$(S)\OC(x+\G)$ is a subgroup of $(S)\OC(\G)$ of index $2$ given by
				\begin{alignat*}{2}
					\SOC(x+\G)&=\{R_q\in \SOC(\G):\nrq\not\equiv 2\imod{4}\},\text{ and}\\
					\OC(x+\G)&=\SOC(x+\G)\cup\{T_q:\nrq\equiv 2\imod{4}\}.
				\end{alignat*}
				  If $f_{x+\G}(m)$, $\hat{f}_{x+\G}(m)$, and $\hat{F}_{x+\G}(m)$ denote the number of CSLs, coincidence rotations, and linear  
				  coincidence isometries of $x+\G$ of index $m$, respectively, then $f_{x+\G}(m)=f_{\Z^3}(m)$, $\hat{f}_{x+\G}(m)=12f_{x+\G}(m)$, and 
				  $\hat{F}_{x+\G}(m)=24f_{x+\G}(m)$.
			\end{lemma}
			\begin{proof}
				The explicit expression for $\SOC(x+\G)$ was obtained from the computations preceding the lemma.  Similar calculations yield $\OC(x+\G)$.
				
				Now, $R\in \SOC(x+\G)$ if and only if $R$ is parametrized by a quaternion $q$ with $\nrq=2^m\alpha$, where $m$ is an even integer and $\alpha$ is odd.  
				Similarly, the coincidence reflection $T\in \OC(x+\G)$ if and only if $T$ is parametrized by a quaternion $q$ with $\nrq=2^n\beta$, where $n$ and $\beta$ 
				are odd integers.  With these two criteria, one concludes by going through all the possible cases that $(S)\OC(x+\G)$ is closed under composition.  Hence, 
				by Proposition~\ref{OCgroup}, $(S)\OC(x+\G)$ is a group.
				
				It follows then from Proposition~\ref{countCSLshift} that $f_{x+\G}(m)=f_{\Z^3}(m)$.  Furthermore, expressions for $\hat{f}_{x+\G}(m)$ and 
				$\hat{F}_{x+\G}(m)$ follow from the fact that there are twelve symmetry rotations $R_q$ with ${\nrq\not\equiv 2\imod{4}}$, and 12 rotoreflection symmetries 
				$T_q$ with $\nrq\equiv 2\imod{4}$, respectively.
			\end{proof}

		Finally, applying the same technique used in computing for $\SOC(x+\G)$ yields that neither $x$ nor $Rx$ are in $\G+R\G$ for all $R\in \OC(\G)$.  
		Theorem~\ref{genthm}, together with Lemma~\ref{OCshiftedfcc}, brings about the following solution of the coincidence problem for the diamond packing.

		\begin{thm}\label{OCdiamond}
			Let $\G$ be the f.c.c.~lattice 
			\[
				\G=2\Im\L\cup [(1,1,0)+2\Im\L]
				\cup[(0,1,1)+2\Im\L]\cup[(1,0,1)+2\Im\L],
			\]
			and $D_3^+$ be the diamond packing $D_3^+=\G\cup(x+\G)$, where $x=(1/2)(1,1,1)$.  Then the group of linear coincidence isometries of $D_3^+$ is $\OC(\G)$.  
			In particular, $R=R_q\in \SOC(\G)$ is a coincidence rotation of $D_3^+$ with
			\begin{enumerate}[\rm(i)]
				\item $D_3^+(R)=\G(R)$ and $\S_{D_3^+}(R)=2\S_{\G}(R)=\nrq$ if $\nrq\equiv 2\imod{4}$.
	
				\item $D_3^+(R)=\G(R)\cup[(x+\ell)+\G(R)]$, where $\ell\in(Rx-x+R\G)\cap\G$, and 
				\[\S_{D_3^+}(R)=\S_{\G}(R)=\begin{cases}
					\nrq, &\text{if }\nrq\text{ is odd}\\
					(1/4)\nrq, &\text{if }\nrq\equiv 0\imod{4}.
				\end{cases}\]			
			\end{enumerate}
			Also, $T=T_q\in \OC(\G)\setminus \SOC(\G)$ is a coincidence rotoreflection of $D_3^+$ with
			\begin{enumerate}[\rm(i)]
				\item $D_3^+(T)=\G(T)$ and 
				\[\S_{D_3^+}(T)=2\S_{\G}(T)=\begin{cases}
					2\nrq, &\text{if }\nrq\text{ is odd}\\
					(1/2)\nrq, &\text{if }\nrq\equiv 0\imod{4}.
				\end{cases}\]
	
				\item $D_3^+(T)=\G(T)\cup[(x+\ell)+\G(T)]$, where $\ell\in(Tx-x+T\G)\cap\G$, and $\S_{D_3^+}(T)=\S_{\G}(T)=(1/2)\nrq$ if $\nrq\equiv 2\imod{4}$.
			\end{enumerate}
	
			If $f_{D_3^+}(m)$ is the number of CSPs of $D_3^+$ of index $m$, then $f_{D_3^+}$ is multiplicative and for primes $p$ and $r\in\N$, 
			\[f_{D_3^+}(p^r)=\begin{cases}
			1, &\text{if }p^r=2\\
			0, &\text{if }p=2\text{ and }r>1\\
			(p+1)p^{r-1}, &\text{otherwise}.
			\end{cases}\]
			The Dirichlet series generating function for $f_{D_3^+}(m)$ reads
			\begin{alignat}{2}\label{dirdia}
				\Phi_{D_3^+}(s)&=\sum_{m=1}^{\infty}\frac{f_{D_3^+}(m)}{m^s}=(1+2^{-s})\cdot\Phi_{\Z^3}(s)=\frac{\zeta_{\J}(s/2)}{\zeta(2s)}\\\notag
				&=1+\tfrac{1}{2^s}+\tfrac{4}{3^s}+\tfrac{6}{5^s}+\tfrac{4}{6^s}+\tfrac{8}{7^s}+\tfrac{12}{9^s}+\tfrac{6}{10^s}
				+\tfrac{12}{11^s}+\tfrac{14}{13^s}+\tfrac{8}{14^s}+\tfrac{24}{15^s}+\tfrac{18}{17^s}+\cdots\,.
			\end{alignat}
			Finally, the number of CSPs of $D_3^+$ with index at most $N$ is asymptotically $9N^2/(2\pi^2)$.
		\end{thm} 
		These results reflect nicely the special shelling structure of $D_3^+$. Observe that the points of $\G$ lie on shells of radius $r^2\equiv 0\imod{4}$
		(where $r^2$ must not be of the form $4^n(8k+7)$, see~\cite{HW08,E85}) and on shells of radius $r^2\equiv 2\imod{4}$.  On the other hand, the points of $x+\G$
		lie on shells with $4r^2\equiv 3\imod{8}$, see~\cite{HW08,E85}.  Thus, no coincidence isometry of $D_3^+$ can map points of $\G$ onto points of $x+\G$ and vice versa, 
		which leads to $\S_{D_3^+}(R)\geq\S_{\G}(R)$ for all coincidence isometries $R$.  The case $\S_{D_3^+}(R)=\S_{\G}(R)$ corresponds to those $R$ for which there are coincidences in
		shells of both $\G$ and $x+\G$, whereas $\S_{D_3^+}(R)=2\S_{\G}(R)$ holds if there are only coincidences in shells containing points of $\G$.

		Note that $\S_{D_3^+}(R)$ and $\S_{\G}(R)$ are essentially given by the norm $\nrq$ of a Hurwitz quaternion, or, if we view them as vectors in $\R^4$, by the square of the length
		of a vector of the four-dimensional centered hypercubic lattice $D_4$.  We thus expect a connection to the shelling problem of $D_4$, or more precisely, to the root lattice $D_4$
		scaled by a factor of $1/2$, see~\cite{CS99}. Indeed, in~\eqref{dirdia}, $\zeta_{\J}(s)=(1/24)\sum_{0\ne q\in\J}(1/|q|^{4s})$ is the generating function for the number of 
		nonzero right ideals of $\J$, and likewise $24\zeta_{\J}(s/2)=\sum_{0\ne q\in\J}(1/|q|^{2s})=\sum_{n\in\N}(c_{\J}(n)/{n^s})$ is the generating function
		for the number $c_{\J}(n)$ of points of $\J$ with square length $\nrq=n$.  The additional factor $1/\zeta(2s)$ in~\eqref{dirdia} is due to the fact that we only count primitive
		quaternions.

	\section{Outlook}
	
		In this paper, the idea of linear coincidence isometries of lattices was extended to include affine isometries.  Moreover, the coincidence problem for shifted 
		lattices and for crystallographic point packings was formulated in a mathematical setting and was solved for some important examples.  Considering further lattices and crystal
		structures would be interesting.  For applications to quasicrystals, the ideas in this paper should be extended to the $\Z$-module case~\cite{BG13}.  In particular,
		techniques implemented and results obtained in Section~\ref{coincshiftsqlatt} on the coincidences of a shifted square lattice may be generalized to planar
		modules by identifying these modules with rings of cyclotomic integers.  Initial results in this direction can be found in~\cite{L10}.

		The set of affine coincidence isometries of a lattice and the set of linear coincidence isometries of a shifted lattice do not form a group in general.  An 
		investigation of their algebraic structure should prove worthwhile.  It has been shown in~\cite{L10} that both sets are groupoids if and only if they are 
		groups.  An example where the set of coincidence rotations of a shifted lattice fails to form a group is still lacking.  Such an example might be found in 
		three-dimensions, where $\OG(d)$ is not Abelian anymore.
		
	\section*{Acknowledgements}
		The authors would like to thank the anonymous referees for their helpful comments and suggestions.  
		M.J.C.~Loquias would like to thank the German Academic Exchange Service (DAAD) for financial support during his stay in Bielefeld.  
		He also acknowledges the Office of the Chancellor of the University of the Philippines Diliman, through the Office of the Vice Chancellor for Research and Development, 
		for funding support through the Ph.D.~Incentive Awards.  This work was supported by the German Research Council (DFG), within the CRC 701.
		
	\bibliographystyle{amsplain}
	\bibliography{iucr}
\end{document}